\documentclass[a4paper,10pt]{article}

\usepackage{cmap}
\usepackage[T1]{fontenc}
\usepackage{lmodern}
\usepackage[margin=1.03in]{geometry}
\usepackage{amsmath,amssymb,amsthm,mathtools}
\numberwithin{equation}{section}
\usepackage{enumitem}
\usepackage{needspace}
\usepackage{microtype}
\usepackage{mathrsfs}
\usepackage[pagebackref]{hyperref}
\allowdisplaybreaks[2]
\hypersetup{
 colorlinks=true,
 linkcolor=blue,
 citecolor=blue,
 urlcolor=blue,
 pdftitle={On the Geometry of Wasserstein Barycenters II: Riemannian Rigidity, Essential Non-Branching, and Finsler Models},
 pdfauthor={Bang-Xian Han and Deng-Yu Liu},
 pdfkeywords={Wasserstein barycenter, curvature-dimension, RCD, Finsler geometry, non-branching}
}
\renewcommand*{\backrefalt}[4]{%
 \ifcase #1\relax
 \or \unskip\nobreak\hspace{.4em}%
     \mbox{\footnotesize\textup{[cited on p.~#2]}}%
 \else \unskip\nobreak\hspace{.4em}%
     \mbox{\footnotesize\textup{[cited on pp.~#2]}}%
 \fi}
\newtheorem{theorem}{Theorem}[section]
\newtheorem{proposition}[theorem]{Proposition}
\newtheorem{lemma}[theorem]{Lemma}
\newtheorem{corollary}[theorem]{Corollary}
\newtheorem{maintheorem}{Theorem}

\theoremstyle{definition}
\newtheorem{definition}[theorem]{Definition}
\theoremstyle{remark}
\newtheorem{remark}[theorem]{Remark}
\newtheorem{example}[theorem]{Example}
\theoremstyle{plain}

\newcommand{\R}{\mathbb R}
\newcommand{\cL}{\mathcal L}
\newcommand{\cP}{\mathcal P}
\newcommand{\Ent}{\operatorname{Ent}}
\newcommand{\Bary}{\operatorname{Bar}}
\newcommand{\Var}{\operatorname{Var}}
\newcommand{\tr}{\operatorname{tr}}
\newcommand{\supp}{\operatorname{spt}}
\newcommand{\dist}{\mathsf d}
\renewcommand{\d}{\,\mathrm d}
\newcommand{\mm}{\mathfrak m}
\newcommand{\Sym}{\operatorname{Sym}}
\newcommand{\argmin}{\operatorname*{argmin}}
\newcommand{\BCD}{\mathrm{BCD}}
\newcommand{\RCD}{\mathrm{RCD}}
\newcommand{\CD}{\mathrm{CD}}
\newcommand{\CDe}{\mathrm{CD}^{e}}
\newcommand{\CDs}{\mathrm{CD}^{*}}
\newcommand{\Ch}{\operatorname{Ch}}
\newcommand{\Pp}{\mathcal P_2}
\newcommand{\mh}{\widehat{\mathfrak m}}

\title{{\bfseries On the Geometry of Wasserstein Barycenters II:}\\[0.35em]
{\Large\normalfont Riemannian Rigidity, Essential Non-Branching, and Finsler Models\par}}
\author{Bang-Xian Han\thanks{School of Mathematics, Shandong University, Jinan 250100, China. Email: hanbx@sdu.edu.cn.}
\and Deng-Yu Liu\thanks{School of Mathematical Sciences, University of Science and Technology of China, Hefei 230026, China. Email: yzldy@mail.ustc.edu.cn.}}
\date{}

\begin{document}
\maketitle

\begin{abstract}
Wasserstein barycenters provide a notion of weighted mean for probability
measures on a metric space.  The barycenter curvature-dimension condition (BCD)
is defined via a Jensen-type inequality for entropy at Wasserstein barycenters.
We prove that $\BCD(K,N)$ is equivalent to $\RCD(K,N)$
for $K\in\R$ and $1<N<\infty$.  This gives a new characterization of
the Riemannian curvature-dimension condition.

We also introduce an almost BCD condition that allows an additive
error in the entropy inequality.  It is stable under measured Gromov--Hausdorff
convergence and implies essential non-branching when the error is
sufficiently small.  The resulting classes contain non-Riemannian Minkowski
spaces and address an open problem posed by Ambrosio in his 2018 ICM survey.
We further bound the failure of the parallelogram identity for cotangent norms
in terms of this error.  For smooth strongly convex Minkowski norms, we
obtain upper and lower bounds for the error.  These bounds have the same
asymptotic order near Euclidean norms.
\end{abstract}



\begingroup
\hypersetup{linkcolor=black}
\tableofcontents
\endgroup

\section{Introduction}

\paragraph{Background and motivation.}

McCann introduced displacement convexity for functionals on the
Wasserstein space \cite{McCannDisplacement}.  The Lott--Sturm--Villani
curvature-dimension condition formulates lower Ricci curvature bounds
through displacement convexity inequalities for the relative entropy
\cite{LottVillani,SturmI,SturmII}. These inequalities compare entropy
along a Wasserstein geodesic with its values at the two endpoints.

Write $W_2$ for the quadratic Wasserstein distance.
For probability measures $\mu_1,\ldots,\mu_k$ with finite second moment
and positive weights $\lambda_1,\ldots,\lambda_k$ summing to
one, a Wasserstein barycenter minimizes
$
\nu\mapsto\sum_i\lambda_iW_2^2(\nu,\mu_i)
$
over probability measures $\nu$ with finite second moment.
The existence, regularity, and geometry of Wasserstein barycenters
have been studied in
\cite{AguehCarlier,LeGouicLoubes,KimPass,MaWassersteinBarycenter}.

For two measures $\mu_0,\mu_1$ with weights $1-t,t$, where $0<t<1$,
the minimizers in the corresponding Wasserstein barycenter problem are precisely the $t$-intermediate points of Wasserstein geodesics
from $\mu_0$ to $\mu_1$.  Thus displacement convexity is a two-point
entropy inequality at barycenters.
The barycenter curvature-dimension condition (BCD) extends the
two-point entropy inequality to finite families of measures.
It requires at least one Wasserstein barycenter to satisfy a Jensen-type
entropy inequality with curvature-dependent coefficients
\cite{HanLiuZhuBCD}.
The survey \cite{HanLiuZhuSurvey} presents this condition on extended
metric measure spaces.

The distinction between two-point and multipoint inequalities is
apparent in normed spaces. Every $n$-dimensional normed space equipped with
Lebesgue measure satisfies $\CD(0,n)$, even when the norm is not induced by
an inner product \cite{Villani2009,OhtaFinslerInterpolation}. In contrast,
we showed in \cite{HanLiuRigidity} that Jensen convexity
of Boltzmann entropy at Wasserstein barycenters forces the norm on a
finite-dimensional space to be induced by an inner product.  In the setting of metric measure
spaces, Ambrosio, Gigli, and Savar\'e
introduced the Riemannian curvature-dimension condition (RCD) by combining
the Lott--Sturm--Villani curvature-dimension condition with the quadraticity
of the Cheeger energy
\cite{AGSRCD}. The latter property is called infinitesimal Hilbertianity
\cite[Section~4.3]{GigliDifferential}. These results naturally raise the
question of whether finite-dimensional BCD characterizes the Riemannian
curvature-dimension condition on general metric measure spaces.

A key issue in extending rigidity results from normed spaces to general
metric measure spaces is the possible branching of geodesics in the
underlying space.  For optimal transport between absolutely continuous
measures, the relevant notion is essential non-branching
\cite{RajalaSturm}.
Essential non-branching plays a fundamental role in the existence and
uniqueness of optimal transport maps, which is a central problem in
optimal transport theory
\cite{GigliOptimalMaps,GigliRajalaSturm}.
Kell showed that, under suitable non-degeneracy assumptions, essential
non-branching is equivalent to the existence of a unique optimal coupling
induced by a map whenever the source measure is absolutely continuous
\cite{KellTransport}.

Beyond exact rigidity, we ask whether barycentric entropy inequalities
still imply essential non-branching when a small additive error is allowed.
We also ask whether the resulting condition is stable under measured
Gromov--Hausdorff convergence.
This connects our investigation to an open problem posed by Ambrosio
in his 2018 ICM survey: whether there is a larger non-Riemannian stable
class of essentially non-branching spaces
\cite[Section~9, first open problem, p.~331]{AmbrosioICM}.
In this direction,
Magnabosco proved stability of the very strict $\CD(K,\infty)$ condition
under additional metric assumptions, with the limit satisfying weak
essential non-branching \cite[Section~2]{MagnaboscoStability}.

\subsection*{Setting and conventions}

Unless stated otherwise, all metric measure spaces are complete, separable,
and geodesic, with locally finite Radon measures of full support.
For such a space $(X,\dist,\mm)$, write $\cP(X)$ for its Borel probability
measures and $\Pp(X)$ for
those with finite second moment, equipped with $W_2$.
For $\mu,\nu\in\cP(X)$, let $\Pi(\mu,\nu)$ be the set of their couplings.
We write $B_r(x)$ for an open metric ball and $\mu|_A$ for the restriction
of a measure to a Borel set $A$.

For $\mu\in\cP(X)$, define
\[
 \Ent_\mm(\mu):=
 \begin{cases}
 \displaystyle\int_X\rho\log\rho\d\mm,
   &\text{if $\mu=\rho\mm$},\\[3pt]
 +\infty,&\text{otherwise},
 \end{cases}
\]
where $0\log0=0$.  If both the positive and negative parts of
$\rho\log\rho$ have infinite integral, we set $\Ent_\mm(\mu):=+\infty$.
The entropy may take the value $-\infty$.  Write
\(\mathrm D(\Ent_\mm)
:=\{\mu\in\Pp(X):\Ent_\mm(\mu)\in\R\}\),
and throughout, ``finite entropy'' means membership in this domain:
the measure has finite second moment and its entropy is a finite real number.
The domain $\mathrm D(\Ent_\mm)$ is $W_2$-dense in $\Pp(X)$.  Finitely supported
probability measures are $W_2$-dense by \cite[Theorem~6.18]{Villani2009}.
For each atom, the normalized restriction of $\mm$ to a sufficiently small
ball gives an approximating probability measure with finite entropy.

\paragraph{Curvature-dimension conditions.}

For $1<N<\infty$, put \(U_N(\mu):=\exp(-\Ent_\mm(\mu)/N)\).
For $\kappa\in\mathbb R$, define $s_\kappa$
on $[0,\infty)$ if $\kappa\leq 0$, and on
$[0,\pi/\sqrt{\kappa})$ if $\kappa>0$, by
\[
s_\kappa(r)=
\begin{cases}
\kappa^{-1/2}\sin(\sqrt\kappa r),&\kappa>0,\\
r,&\kappa=0,\\
(-\kappa)^{-1/2}\sinh(\sqrt{-\kappa}r),&\kappa<0
\end{cases}.
\]
For $D\ge0$ and $t\in[0,1]$, the distortion coefficients are defined by
\cite[Definition~2.1]{EKS}
\[
 \sigma_\kappa^{(t)}(D)=
 \begin{cases}
 s_\kappa(tD)/s_\kappa(D),&\kappa D^2\ne0\text{ and }\kappa D^2<\pi^2,\\
 t,&\kappa D^2=0,\\
 +\infty,&\kappa D^2\ge\pi^2.
 \end{cases}
\]
For $K\in\R$, the entropic curvature-dimension condition $\CDe(K,N)$
of \cite[Definition~3.1]{EKS} requires that every pair
$\mu_0,\mu_1\in\mathrm D(\Ent_\mm)$ can be joined by a constant-speed
$W_2$-geodesic $(\mu_t)_{t\in[0,1]}\subset\mathrm D(\Ent_\mm)$ satisfying
\begin{equation}\label{eq:cde-definition}
 U_N(\mu_t)\ge
 \sigma_{K/N}^{(1-t)}(D)U_N(\mu_0)
 +\sigma_{K/N}^{(t)}(D)U_N(\mu_1),
 \qquad D=W_2(\mu_0,\mu_1).
\end{equation}
We call $\CDe(K,N)$ \emph{strong} if \eqref{eq:cde-definition} holds along every
such geodesic.

The dimension-free condition $\CD(K,\infty)$ retains the curvature bound
$K$ without a finite upper dimension bound.  It requires $K$-convexity of
$\Ent_\mm$ along one $W_2$-geodesic between each pair
$\mu_0,\mu_1\in\Pp(X)$ that are absolutely continuous with respect to
$\mm$.  Its strong form requires the inequality
along every such geodesic \cite{RajalaSturm}.
The finite-dimensional condition $\CDe(K,N)$ becomes weaker as $N$
increases and implies $\CD(K,\infty)$ by passage to the limit
$N\to\infty$ \cite[Lemmas~2.12 and~3.2]{EKS}.

The reduced condition $\CDs(K,N)$, introduced by Bacher--Sturm
\cite{BacherSturm}, requires the R\'enyi-entropy inequality with coefficients
$\sigma_{K/N'}$, for every $N'\ge N$ and every pair of compactly supported
absolutely continuous probability measures \cite[Definition~3.9]{EKS}.
We call $\CDs(K,N)$
\emph{strong} if this inequality holds along every $W_2$-geodesic between
such endpoint measures.
Strong $\CDe(K,N)$ and strong
$\CDs(K,N)$ are equivalent \cite[Corollary~3.13]{EKS}.  The
condition $\RCD^*(K,N)$ is $\CDe(K,N)$ together with infinitesimal
Hilbertianity \cite[Definition~3.16 and Theorem~3.17]{EKS}, whereas
$\RCD(K,N)$ uses the unreduced Lott--Sturm--Villani $\CD(K,N)$ condition.
Cavalletti--Milman proved that $\RCD^*(K,N)$ and $\RCD(K,N)$ are equivalent
for finite reference measures \cite[Corollary~13.7]{CavallettiMilman}, and
Li extended the result to locally finite reference measures
\cite[Corollary~1.2]{LiGlobalization}.  We write $\RCD(K,N)$ for this common
class.  The notation $\CD(K,N)$ remains
reserved for the unreduced condition and $\CDs(K,N)$ for the reduced one.

\paragraph{Barycenter curvature-dimension condition.}

Let $\Omega=\sum_{i=1}^k\lambda_i\delta_{\mu_i}$ be a finitely supported
probability measure on $\Pp(X)$, with $\lambda_i>0$ and
$\sum_i\lambda_i=1$.  Set
\[
\Var(\Omega):=\inf_{\nu\in\Pp(X)}\sum_i\lambda_iW_2^2(\nu,\mu_i),
\qquad
\Bary(\Omega):=\argmin_{\nu\in\Pp(X)}\sum_i\lambda_iW_2^2(\nu,\mu_i).
\]
Put $c_\kappa=s_\kappa'$ and
$t_\kappa=s_\kappa/c_\kappa$ wherever
$c_\kappa\ne 0$.
We interpret the coefficients
$r/s_\kappa(r)$ and $r/t_\kappa(r)$
by continuity at $r=0$, where both equal $1$.
The latter coefficient is always understood as
\[
\frac{r}{t_\kappa(r)}
:=\frac{r c_\kappa(r)}{s_\kappa(r)},
\]
so it remains well-defined at zeros of
$c_\kappa$.

When $K>0$, the barycenter $\bar\mu$
required in the $\mathrm{BCD}(K,N)$ condition
must satisfy
\[
r_i=W_2(\bar\mu,\mu_i)
<\pi\sqrt{\frac{N}{K}}
\qquad\text{for every }i.
\]
Thus all coefficients in the inequality below
are evaluated within their prescribed domains.

The finite-dimensional condition $\BCD(K,N)$ is introduced in
\cite[Definition~1.9]{HanLiuZhuBCD}.  That definition applies to arbitrary
finitely supported probability measures on $\Pp(X)$.   Let
$\Omega=\sum_i\lambda_i\delta_{\mu_i}$ be as above, with
$\mu_i\in\mathrm D(\Ent_\mm)$.  The condition requires the existence of a barycenter
$\bar\mu\in\Bary(\Omega)$ for which the following inequality holds, where
$r_i=W_2(\bar\mu,\mu_i)$:
\begin{equation}\label{eq:finite-bcd}
 \sum_i\lambda_i\frac{r_i}{s_{K/N}(r_i)}U_N(\mu_i)
 \le U_N(\bar\mu)
 \sum_i\lambda_i\frac{r_i}{t_{K/N}(r_i)}.
\end{equation}
At $K=0$, \eqref{eq:finite-bcd} reduces to
$U_N(\bar\mu)\ge\sum_i\lambda_iU_N(\mu_i)$.
\subsection*{Main results}

The first main theorem identifies the finite-dimensional BCD condition with
the Riemannian curvature-dimension condition.  It provides a new multipoint
characterization of RCD spaces.

\begin{maintheorem}\label{thm:main-rigidity}
For every $K\in\R$ and $1<N<\infty$,
\[
 \BCD(K,N)\quad\Longleftrightarrow\quad\RCD(K,N).
\]
\end{maintheorem}

The implication from RCD to BCD follows from a gradient-flow method for
Wasserstein barycenters \cite[Theorem~1.5]{HanLiuZhuBCD}.  For the reverse
implication, we first pass to a dimension-free entropy inequality
in Proposition~\ref{prop:finite-to-infinite}.  Theorem~\ref{thm:enb} then gives
essential non-branching.  The two-point barycenter inequality yields
$\CDe(K,N)$ and local doubling.  The dimension-free curvature condition
then gives local Poincar\'e (PI) estimates.
We use a theorem of
Cheeger--Kleiner--Schioppa \cite[Theorem~7.18 and Corollary~7.99]{CKS}.
It gives a metric
submersion from a pointed measured tangent of the PI space onto its
finite-dimensional tangent fiber, equipped with Cheeger's norm.  For each
unit vector in this fiber, it also gives an Alberti representation \cite{AlbertiRankOne} by
complete geodesic lines.  Pointed stability transfers the entropy inequality
to these tangents.  We then adapt the four-marginal argument from our
earlier work \cite{HanLiuRigidity} to transport along these geodesic lines.
This gives the parallelogram identity for the cotangent norm and hence
infinitesimal Hilbertianity.

The non-branching step in this proof is part of a more general result for
a dimension-free barycenter inequality with an additive error.  The
case of zero error gives the essential non-branching needed in
Theorem~\ref{thm:main-rigidity}.  Allowing a positive error also gives
stable classes containing non-Riemannian models, as in Ambrosio's question.
The definition is inspired by the $(\varepsilon,\delta)$-weak quadruple
condition of \cite[Definition~3.1]{HanYinWeakQuadruple}, where an additive
error is introduced into a comparison inequality.

\begin{definition}[$\delta$-almost BCD]\label{def:almost-bcd}
Fix $K\in\R$ and $\delta\ge0$.  Let
$\Omega=\sum_{i=1}^k\lambda_i\delta_{\mu_i}$ be a finitely supported
probability measure on $\Pp(X)$ with
$\mu_i\in\mathrm D(\Ent_\mm)$.  The space $(X,\dist,\mm)$ satisfies
$\delta$-almost $\BCD(K,\infty)$, written
$\BCD_\delta(K,\infty)$, if every such $\Omega$ has a Wasserstein barycenter
$\bar\mu\in\Bary(\Omega)$ satisfying
\begin{equation}\label{eq:intro-def}
\Ent_\mm(\bar\mu)
\le \sum_i\lambda_i\Ent_\mm(\mu_i)
-\frac K2\Var(\Omega)+\delta.
\end{equation}
\end{definition}
\begin{remark}
	The definition includes both the existence of a Wasserstein barycenter and
	\eqref{eq:intro-def}.  On a proper space, the existence statement alone is
	automatic by \cite[Theorem~2]{LeGouicLoubes}.
	For $\delta=0$, the inequality is the dimension-free BCD inequality,
	restricted here to marginals in $\mathrm D(\Ent_\mm)$.
\end{remark}

To state the infinitesimal conclusion, for a finite-dimensional normed space
$(E^*,\|\cdot\|_*)$ define the normalized parallelogram defect
\[
\mathfrak p_{\mathrm{par}}(E^*):=
\sup_{p,q\in E^*, (p,q)\ne(0,0)}
\frac{\left|
\|p+q\|_*^2+\|p-q\|_*^2-2\|p\|_*^2-2\|q\|_*^2
\right|}
{\max\{\|p+q\|_*^2,\|p-q\|_*^2,
2\|p\|_*^2,2\|q\|_*^2\}}.
\]
On a locally doubling space supporting a local weak $(1,1)$-Poincar\'e
inequality, Cheeger's differentiability theorem gives a finite-dimensional
measurable cotangent bundle
\cite[Theorem~4.38 and Definition~4.42]{CheegerDifferentiability}.
We write $T_x^*X$ for its fiber at $x$, equipped with Cheeger's cotangent
norm.  We describe this norm and its relation to Lipschitz functions in
Section~\ref{sec:tangent-lines}.

\begin{maintheorem}\label{thm:intro-almost}
Suppose that $(X,\dist,\mm)$ satisfies $\BCD_\delta(K,\infty)$.
\begin{enumerate}[label=\textup{(\roman*)},leftmargin=2.4em]
\item For each $\Omega$ as in Definition~\ref{def:almost-bcd}, every
$\bar\mu\in\Bary(\Omega)$ satisfies \eqref{eq:intro-def}.
\item If $\delta<\frac12\log2$, then $(X,\dist,\mm)$ is essentially
non-branching.
\item If, in addition, $X$ is locally doubling and supports a local weak
$(1,1)$-Poincar\'e inequality, then for $\mm$-almost every $x$,
\[
 \mathfrak p_{\mathrm{par}}(T_x^*X)\le8\sqrt\delta.
\]
\end{enumerate}
\end{maintheorem}

The sufficient bound $\delta<\frac12\log2$ in part~\textup{(ii)} comes from
the entropy gap $\log2$ for an equal mixture of two mutually singular
probability measures in the Rajala--Sturm argument
\cite[Step~7]{RajalaSturm}.  We do not determine the optimal range of
$\delta$ for essential non-branching.

Part~\textup{(iii)} bounds the pointwise parallelogram defect by the error in
the entropy inequality and holds for every $\delta\ge0$.
At $\delta=0$, it gives the infinitesimal Hilbertianity needed in
Theorem~\ref{thm:main-rigidity}.
Corollary~\ref{cor:quantitative-cheeger} gives analogous estimates
for the Cheeger energy.

The almost BCD condition is also stable under pointed measured
Gromov--Hausdorff convergence for proper spaces
(Theorem~\ref{thm:pointed-stability}).  To obtain stable non-Riemannian
classes, it therefore remains to find non-Riemannian models whose error
satisfies the bound in part~\textup{(ii)}.  We do this by estimating the
smallest possible error for Minkowski norms.

Let $F$ be a smooth strongly convex Minkowski norm on $\R^n$, with
smoothness required only on $\R^n\setminus\{0\}$, and let $F^*$ denote
its dual norm.  Set $h=F^2/2$ and write
$g_v:=\mathrm D^2h(v)$ for $v\ne0$; this is the fundamental tensor of the
Minkowski structure \cite[Chapter~1]{BaoChernShen}.
  With the convention $\inf\varnothing=+\infty$, set
\[
\Delta_{\mathrm{BCD}}(F)
:=\inf \Big\{\delta\ge0:(\R^n,F,\cL^n)\text{ satisfies }
\BCD_\delta(0,\infty)\Big\}.
\]
Theorem~\ref{thm:intro-almost} bounds $\Delta_{\mathrm{BCD}}(F)$ from below
in terms of $\mathfrak p_{\mathrm{par}}(F^*)$.  For an upper bound, we use
the Jacobian approach to barycenter entropy inequalities
\cite[Section~7]{AguehCarlier}.  For a Minkowski cost, the Jacobian
identities involve the fundamental tensors.  The resulting entropy estimate
contains the difference between the two sides of Jensen's inequality for
the concave function $A\mapsto\log\det A$;
see \cite[Section~4.2]{BhatiaPositive} for the determinant inequalities.
We therefore define
\begin{equation}\label{eq:intro-cf}
c_F(F):=
\sup_{\substack{k\ge2,\ \lambda_i>0,\ \sum_i\lambda_i=1\\v_i\ne0}}
\left\{
\log\det\left(\sum_i\lambda_i g_{v_i}\right)
-\sum_i\lambda_i\log\det g_{v_i}
\right\}.
\end{equation}
This gap vanishes when the fundamental tensor is independent of direction.
It is invariant under invertible linear changes of coordinates.  
The supremum may also be taken over all probability measures on the compact
family of fundamental tensors; see \eqref{eq:cf-prob}.
Proposition~\ref{prop:smooth-entropy} derives the entropy error from this
matrix expression, and Theorem~\ref{thm:intro-finsler} gives the resulting
upper bound.

Write $|\cdot|$ for the Euclidean norm and $S^{n-1}$ for its unit sphere.
We next consider perturbations of the Euclidean norm.
Let $\psi\in C^\infty(S^{n-1})$ be even.  For sufficiently small $|\tau|$, set
\[
F_\tau(v)^2=|v|^2\left(1+\tau\psi\left(\frac v{|v|}\right)\right)
\quad(v\ne0),
\qquad F_\tau(0)=0.
\]
If $\psi$ is the restriction of a quadratic form, then $F_\tau$ is induced
by an inner product.  We consider the case in which $\psi$ is not the
restriction of a quadratic form.

\begin{maintheorem}\label{thm:intro-finsler}
Every smooth strongly convex Minkowski norm $F$ on $\R^n$ satisfies
\[
 \frac1{64}\mathfrak p_{\mathrm{par}}(F^*)^2
 \le\Delta_{\mathrm{BCD}}(F)
 \le c_F(F).
\]
The space $(\R^n,F,\cL^n)$ satisfies
$\BCD_{c_F(F)}(0,\infty)$.  The three conditions
\[
 \mathfrak p_{\mathrm{par}}(F^*)=0,
 \qquad
 \Delta_{\mathrm{BCD}}(F)=0,
 \qquad
 c_F(F)=0
\]
are equivalent, and they hold if and only if $F$ is induced by an inner
product.
For every fixed $\psi$ as above that is not the restriction of a quadratic form,
there exist $0<c_\psi\le C_\psi<\infty$ and $\tau_0>0$ such that
\[
c_\psi\tau^2
\le \Delta_{\mathrm{BCD}}(F_\tau)
\le c_F(F_\tau)
\le C_\psi\tau^2
\qquad (|\tau|<\tau_0).
\]
\end{maintheorem}

Theorem~\ref{thm:intro-finsler} includes non-Riemannian examples and proves
that the exponent in Theorem~\ref{thm:intro-almost}\textup{(iii)} is sharp.

\subsection*{Beyond small entropy errors}

Theorems~\ref{thm:intro-almost}, \ref{thm:pointed-stability}, and
\ref{thm:intro-finsler} give stable essentially non-branching classes
containing non-Riemannian Minkowski spaces.  This addresses an
open problem posed by Ambrosio \cite[Section~9, p.~331]{AmbrosioICM} in the BCD setting.  Can such a class include all smooth
reversible Finsler models satisfying a curvature-dimension bound?
A class closed under pointed measured Gromov--Hausdorff convergence cannot
do so while retaining essential non-branching.  This already fails for
flat Minkowski planes.
Indeed, the smooth strongly convex norms
\[
 F_j(v)^2=(v_1^{2j}+v_2^{2j})^{1/j}+j^{-1}|v|^2,
 \qquad v\in\R^2,\quad j\ge1,
\]
converge locally uniformly to $\|\cdot\|_\infty$.  Each space
$(\R^2,F_j,\cL^2,0)$ is non-branching and satisfies $\CD(0,2)$
\cite[Remark~8.2(b)]{OhtaFinslerInterpolation}, but the limit is not
essentially non-branching.  

To include a broader range of finite-dimensional Finsler models, one may
instead consider closed subclasses defined by quantitative non-branching
bounds.  By contrast, the almost BCD approach  also applies to infinite-dimensional
spaces, consistent with the dimension-free framework of Ambrosio's open problem.  We describe two possible forms of quantitative non-branching bounds.

One possible approach is based on the excess cost in a two-point Wasserstein
barycenter problem.  For $\mu_0,\mu_1,\nu\in\Pp(X)$ and $0<t<1$, set
\[
 \mathcal E_t(\nu\mid\mu_0,\mu_1)
 :=(1-t)W_2^2(\nu,\mu_0)+tW_2^2(\nu,\mu_1)
   -t(1-t)W_2^2(\mu_0,\mu_1).
\]
This quantity is nonnegative and vanishes exactly at the barycenters of
$(1-t)\delta_{\mu_0}+t\delta_{\mu_1}$.  A uniform positive lower bound for
$\mathcal E_t(\nu_0\mid\mu_0,\mu_1)+\mathcal E_t(\nu_1\mid\mu_0,\mu_1)$
when $W_2(\nu_0,\nu_1)\ge s>0$ gives a quantitative form of uniqueness.
The bound may depend on $s$, a common upper bound on the endpoint entropies,
a radius containing their supports, and a positive lower bound on
$\min\{t,1-t\}$.  It must hold uniformly over all competitors
$\nu_0,\nu_1\in\Pp(X)$ with $W_2(\nu_0,\nu_1)\ge s$.
The uniqueness results in \cite[Theorem~1.1]{CavallettiMondino} and
\cite[Theorem~1.1]{GigliRajalaSturm} motivate such estimates.

Alternatively, one can formulate a quantitative non-branching condition
directly in terms of the metric.  For $x,z,y_0,y_1\in X$, define
\[
\mathcal F(x,z;y_0,y_1)
:=\sum_{i=0}^1\left[
\tfrac12\dist(x,z)^2+\tfrac12\dist(z,y_i)^2
-\tfrac14\dist(x,y_i)^2\right].
\]
This quantity is nonnegative.  It vanishes exactly when $z$ is a midpoint
of both $(x,y_0)$ and $(x,y_1)$.  If $y_0\ne y_1$, the resulting
geodesics branch at $z$.
On a proper non-branching space, compactness gives a positive lower bound
for $\mathcal F$ on quadruples in a fixed closed ball with
$\dist(y_0,y_1)\ge s>0$.
Prescribing these lower bounds gives classes stable under pointed
Gromov--Hausdorff convergence.
Such bounds hold for smooth strongly convex Minkowski spaces.
For fixed finite $N$, they can also be obtained uniformly for
$\RCD(K,N)$ spaces from Deng's non-branching theorem
\cite{DengNonbranching} and compactness \cite{GigliMondinoSavare}.

\subsection*{Organization}

Section~\ref{sec:almost} proves essential non-branching and stability
under the almost BCD condition.  Section~\ref{sec:exact} proves
quantitative infinitesimal Hilbertianity and Riemannian rigidity.
Section~\ref{sec:finsler} studies the almost BCD condition on Minkowski
spaces.

\section{Almost BCD and essential non-branching}\label{sec:almost}

\subsection{Volume growth and entropy}

Let $\operatorname{Geo}(X)$ be the space of constant-speed geodesics
$\gamma:[0,1]\to X$, equipped with the uniform distance, and let
$e_t(\gamma)=\gamma_t$ be the evaluation map.
An optimal dynamical plan between $\mu_0,\mu_1\in\Pp(X)$ is a probability
measure $\boldsymbol\pi$ on $\operatorname{Geo}(X)$ such that
$(e_0,e_1)_\#\boldsymbol\pi$ is an optimal coupling of $\mu_0$ and $\mu_1$.
Every $W_2$-geodesic has the representation
$\mu_t=(e_t)_\#\boldsymbol\pi$ for such a plan, and conversely
\cite[Theorem~3.10]{AmbrosioGigliGuide}.

\begin{lemma}\label{lem:entropy-regularity}
If $(X,\dist,\mm)$ satisfies $\BCD_\delta(K,\infty)$, then, for every
$o\in X$, there are constants $C,c>0$ such that
\begin{equation}\label{eq:gaussian-volume-growth}
 \mm(B_R(o))\le C e^{cR^2}\qquad~~\forall R>0.
\end{equation}
Thus $\int_Xe^{-a\dist^2(x,o)}\d\mm(x)<\infty$ for all sufficiently
large $a$.  The entropy $\Ent_\mm$ is sequentially lower semicontinuous under
$W_2$-convergence, and $\Ent_\mm(\mu)>-\infty$ for every $\mu\in\Pp(X)$.
\end{lemma}

\begin{proof}
Choose $0<r<s$ such that
$0<\mm(B_r(o))<\infty$ and $\mm(\overline B_s(o))<\infty$.
Fix $R\ge s$ and a Borel set $A\subset B_R(o)$ with
$0<\mm(A)<\infty$.  Set
\[
 \mu_0:=\frac{\mm|_{B_r(o)}}{\mm(B_r(o))},
 \qquad
 \mu_1:=\frac{\mm|_A}{\mm(A)},
 \qquad
 t:=\frac{s-r}{R+r}.
\]
 Apply $\BCD_\delta(K,\infty)$ to
$(1-t)\delta_{\mu_0}+t\delta_{\mu_1}$ and denote the resulting barycenter by
$\bar\mu_t$.  A two-point Wasserstein barycenter is a time-$t$ point of a
$W_2$-geodesic.  Its dynamical representation
\cite[Theorem~3.10]{AmbrosioGigliGuide} shows that
$\supp\bar\mu_t\subset\overline B_s(o)$, since
\[
 \dist(o,\gamma_t)
 \le \dist(o,\gamma_0)+t\dist(\gamma_0,\gamma_1)
 \le r+t(r+R)=s.
\]
The almost BCD inequality gives $\Ent_\mm(\bar\mu_t)<+\infty$.  Since
$\bar\mu_t$ is supported in the finite-measure set $\overline B_s(o)$,
Jensen's inequality yields
$\Ent_\mm(\bar\mu_t)\ge-\log\mm(\overline B_s(o))$.
Moreover,
$\Var((1-t)\delta_{\mu_0}+t\delta_{\mu_1})
=t(1-t)W_2^2(\mu_0,\mu_1)$.  Hence, with
$K_-:=\max\{-K,0\}$,
\begin{align*}
 \log\mm(A)
 &\le
 \frac{\log\mm(\overline B_s(o))
 -(1-t)\log\mm(B_r(o))+\delta}{t}
 +\frac{K_-}{2}(1-t)(R+r)^2\\
 &\le C_0(1+R^2),
\end{align*}
where $C_0$ is independent of $A$ and $R$.
Since a locally finite measure on a separable metric space is
$\sigma$-finite, we may exhaust $B_R(o)$ by sets $A$ of finite measure.
This proves \eqref{eq:gaussian-volume-growth}; enlarging $C$ also covers
$R<s$.

The volume bound and an annular decomposition show that
$Z_a:=\int_Xe^{-a\dist^2(x,o)}\d\mm(x)<\infty$ whenever $a$ is sufficiently
large.  Define the probability measure
$\mm_a:=Z_a^{-1}e^{-a\dist^2(\cdot,o)}\mm$.  For every $\mu\in\Pp(X)$,
\begin{equation}\label{eq:weighted-entropy-identity}
\Ent_\mm(\mu)
=\Ent_{\mm_a}(\mu)-\log Z_a
-a\int_X\dist^2(x,o)\d\mu(x).
\end{equation}
This is an extended-real identity; the last two terms are finite because
$0<Z_a<\infty$ and $\mu\in\Pp(X)$.
Relative entropy with respect to the probability measure $\mm_a$ is
nonnegative and is lower semicontinuous under narrow convergence by its
variational formula
\cite[Section~4.1.1]{GigliMondinoSavare}.
The second moment in \eqref{eq:weighted-entropy-identity} is continuous under
$W_2$-convergence.  The last two assertions follow.
\end{proof}

\subsection{The inequality for every barycenter}
Although the definition requires the entropy inequality at only one
barycenter, it implies the same inequality at every barycenter.
\begin{proposition}
\label{prop:every-barycenter}
Assume $\BCD_\delta(K,\infty)$.  Let
$\Omega=\sum_i\lambda_i\delta_{\mu_i}$ with
$\mu_i\in\mathrm D(\Ent_\mm)$.  Then every $\nu\in\Bary(\Omega)$ satisfies
\begin{equation}\label{eq:all-barycenters}
 \Ent_\mm(\nu)\le\sum_i\lambda_i\Ent_\mm(\mu_i)
 -\frac K2\Var(\Omega)+\delta.
\end{equation}
\end{proposition}

\begin{proof}
Fix a barycenter $\nu$.  We first prove that its entropy is
finite by approximating it with measures in $\mathrm D(\Ent_\mm)$.  We then
add $\nu$ as an atom, making it the unique barycenter of the enlarged family.

By the density stated in the introduction, choose
$\zeta_j\in\mathrm D(\Ent_\mm)$ converging to $\nu$ in $W_2$.  Choose also
$\varepsilon_j\downarrow0$ so that
$\varepsilon_j|\Ent_\mm(\zeta_j)|\to0$.  Let $\nu_j$ be a barycenter of
$(1-\varepsilon_j)\Omega+\varepsilon_j\delta_{\zeta_j}$ for which the
almost BCD inequality holds.  Set $G(\eta)=\sum_i\lambda_iW_2^2(\eta,\mu_i)$.
Comparison with $\nu$ gives
\[
 (1-\varepsilon_j)G(\nu_j)+\varepsilon_jW_2^2(\nu_j,\zeta_j)
 \le(1-\varepsilon_j)G(\nu)+\varepsilon_jW_2^2(\nu,\zeta_j).
\]
Since $G(\nu_j)\ge G(\nu)$, we obtain
$W_2(\nu_j,\zeta_j)\le W_2(\nu,\zeta_j)$, and hence $\nu_j\to\nu$ in $W_2$.  The
variances are bounded and $\varepsilon_j\Ent_\mm(\zeta_j)=o(1)$.  Lower
semicontinuity from Lemma~\ref{lem:entropy-regularity} gives
$\Ent_\mm(\nu)<+\infty$.  The same lemma gives
$\Ent_\mm(\nu)>-\infty$, so $\nu\in\mathrm D(\Ent_\mm)$.

For $0<\varepsilon<1$, put
$\Omega^\varepsilon=(1-\varepsilon)\Omega+\varepsilon\delta_\nu$.
Since $\nu$ minimizes $G$, for every $\eta\in\Pp(X)\setminus\{\nu\}$ we have
\[
(1-\varepsilon)G(\eta)+\varepsilon W_2^2(\eta,\nu)
> (1-\varepsilon)G(\nu).
\]
Thus $\nu$ is the unique barycenter of $\Omega^\varepsilon$.  Applying
$\BCD_\delta(K,\infty)$ to this family and rearranging gives
\begin{align*}
\Ent_\mm(\nu)
&\le (1-\varepsilon)\sum_i\lambda_i\Ent_\mm(\mu_i)
+\varepsilon\Ent_\mm(\nu)
-\frac K2(1-\varepsilon)\Var(\Omega)+\delta,\\
\Ent_\mm(\nu)
&\le \sum_i\lambda_i\Ent_\mm(\mu_i)-\frac K2\Var(\Omega)
+\frac{\delta}{1-\varepsilon}.
\end{align*}
Letting $\varepsilon\downarrow0$ proves \eqref{eq:all-barycenters}.
\end{proof}

\begin{corollary}
\label{cor:strong-almost}
Assume $\BCD_\delta(K,\infty)$.
For every $W_2$-geodesic $(\mu_t)_{t\in[0,1]}$ with
$\mu_0,\mu_1\in\mathrm D(\Ent_\mm)$,
\begin{equation}\label{eq:almost-strong-cd}
 \Ent_\mm(\mu_t)\le(1-t)\Ent_\mm(\mu_0)+t\Ent_\mm(\mu_1)
 -\frac K2t(1-t)W_2^2(\mu_0,\mu_1)+\delta.
\end{equation}
In particular, let $\boldsymbol\pi$ be an optimal dynamical plan between
measures in $\Pp(X)$, and let $f$ be a bounded nonnegative Borel function
on $\operatorname{Geo}(X)$ with $\int f\d\boldsymbol\pi>0$.
The marginals of $f\boldsymbol\pi/\int f\d\boldsymbol\pi$ satisfy
\eqref{eq:almost-strong-cd} whenever its endpoint marginals belong to
$\mathrm D(\Ent_\mm)$.
After restricting the normalized subplan to $[s,u]\subset[0,1]$, with $s<u$, and
reparametrizing to $[0,1]$, the same assertion holds if the new endpoint
marginals belong to $\mathrm D(\Ent_\mm)$.
\end{corollary}

\begin{proof}
Every time-$t$ point of a Wasserstein geodesic is a barycenter of the
two-point family $(1-t)\delta_{\mu_0}+t\delta_{\mu_1}$, whose variance is
$t(1-t)W_2^2(\mu_0,\mu_1)$.  Proposition~\ref{prop:every-barycenter} applies.
For the last assertion, let $\boldsymbol\pi$ be an optimal dynamical plan
and let $f\ge0$ be a bounded Borel function on $\operatorname{Geo}(X)$ with
$\int f\d\boldsymbol\pi>0$.  The normalized weighted subplan
$f\boldsymbol\pi/\int f\d\boldsymbol\pi$ is optimal, and its marginals form
a $W_2$-geodesic.
Restricting to a time interval and reparametrizing it affinely to $[0,1]$
gives another $W_2$-geodesic.
The first assertion applies.
\end{proof}

The next example shows that a two-point entropy inequality with a finite
additive error need not imply an almost BCD condition with any finite error.
The singular barycenter on a tripod also appears in
\cite{BertrandMaMetricGraph}.
\begin{example}\label{ex:tripod}
Let $X$ be the metric tree obtained by identifying one endpoint of three
unit intervals, and let $\mm$ be its length measure.  Every $W_2$-geodesic
$(\mu_t)$ with $\mu_0,\mu_1\in\mathrm D(\Ent_\mm)$ satisfies
\begin{equation}\label{eq:tripod-two-point}
 \Ent_\mm(\mu_t)\le(1-t)\Ent_\mm(\mu_0)+t\Ent_\mm(\mu_1)+2\log3.
\end{equation}
Nevertheless, $(X,\dist,\mm)$ does not satisfy
$\BCD_\delta(K,\infty)$ for any $K\in\R$ and any finite $\delta\ge0$.
\end{example}

\begin{proof}
Let $o$ be the common endpoint.  Choose $0<a<b<\min\{2a,1\}$ and let
$\nu_i$ be normalized length measure on the radial interval $[a,b]$ in the
$i$-th branch, for $i=1,2,3$.  For points $y_i$ in these intervals, put
$s_i=\dist(o,y_i)$.  If $x$ lies in the $j$-th branch and
$r=\dist(o,x)$, then
\[
 \frac13\sum_{i=1}^3\bigl(\dist^2(x,y_i)-\dist^2(o,y_i)\bigr)
 =r^2+\frac{2r}{3}\left(\sum_{i\ne j}s_i-s_j\right)
 \ge r^2+\frac{2r}{3}(2a-b).
\]
For any $\nu\in\Pp(X)$, glue optimal couplings from $\nu$ to the three
$\nu_i$ over their common marginal $\nu$ and integrate this inequality.
The barycenter functional at $\nu$ is therefore at least its value at
$\delta_o$ plus $\int_X\dist^2(x,o)\d\nu(x)$.
Thus $\delta_o$ is the unique Wasserstein
barycenter of this family with equal weights.  The three measures have finite
entropy, whereas $\Ent_\mm(\delta_o)=+\infty$ because $\mm(\{o\})=0$.
The variance is finite, so no finite $K$ and $\delta$ can satisfy the
almost BCD inequality for this family.

To prove \eqref{eq:tripod-two-point}, let $\boldsymbol\pi$ be an optimal
dynamical plan representing $(\mu_t)$.  Split it according to the ordered
pair $(i,j)$ of branches containing its endpoints.  Discard the pieces of
zero mass, and write $a_{ij}$ for the masses of the remaining pieces and
$\mu_u^{ij}$ for their normalized marginals.  Each normalized subplan is
optimal and is concentrated on geodesics in an interval formed by one or
two branches.  Displacement convexity on the line
\cite[Theorem~2.2]{McCannDisplacement} gives
\[
 \Ent_\mm(\mu_t^{ij})
 \le(1-t)\Ent_\mm(\mu_0^{ij})+t\Ent_\mm(\mu_1^{ij}).
\]
At either endpoint $u=0,1$, the density of
$\mu_u=\sum_{i,j}a_{ij}\mu_u^{ij}$ is at least $a_{ij}$ times the density
of $\mu_u^{ij}$.  Taking logarithms, integrating against each $\mu_u^{ij}$,
and summing with weights $a_{ij}$ gives
\[
 \sum_{i,j}a_{ij}\Ent_\mm(\mu_u^{ij})
 \le\Ent_\mm(\mu_u)-\sum_{i,j}a_{ij}\log a_{ij}.
\]
Average the displacement convexity inequalities and use convexity of
$\Ent_\mm$ under mixtures at time $t$.  Since there are at most nine
pieces, $-\sum_{i,j}a_{ij}\log a_{ij}\le\log9=2\log3$, which proves
\eqref{eq:tripod-two-point}.
\end{proof}

\subsection{Essential non-branching}\label{2.3}

A family of constant-speed minimizing geodesics on $[0,1]$ is
non-branching if any two members that agree on $[0,t]$ for some $0<t<1$
agree on all of $[0,1]$.  The space is essentially non-branching if every
optimal dynamical plan between two absolutely continuous probability
measures in $\Pp(X)$ is concentrated on a non-branching family
\cite[Section~2.2]{RajalaSturm}.  Rajala and Sturm proved that strong
$K$-convexity of the entropy implies essential
non-branching \cite[Theorem~1.1]{RajalaSturm}.  Their proof uses the following
entropy identity.  If $\mu^\uparrow$ and $\mu^\downarrow$ are mutually
singular probability measures with finite entropy, then
\[
	\Ent_\mm\!\left(\frac{\mu^\uparrow+\mu^\downarrow}{2}\right)
	=\frac12\Ent_\mm(\mu^\uparrow)
	+\frac12\Ent_\mm(\mu^\downarrow)-\log2.
\]
Retaining the additive error in their entropy estimates leads to the
condition $2\delta<\log2$.

\begin{theorem}[Essential non-branching]
\label{thm:enb}
If $\delta<\frac12\log 2$, then every $\BCD_{\delta}(K,\infty)$ space is essentially
non-branching.
\end{theorem}

\begin{proof}
Suppose otherwise.  The proof follows
\cite[Proof of Theorem~1.1]{RajalaSturm}.

\textbf{Step 1. Localization and selection of two branches.}
Let $\rho_0,\rho_1$ be the endpoint densities of an optimal dynamical plan that is not
concentrated on a non-branching set.  Multiply the plan by
$(1+\rho_0\circ e_0+\rho_1\circ e_1)^{-1}$ and normalize.
This preserves its null sets and makes both endpoint densities bounded.
The endpoints now have finite entropy by Lemma~\ref{lem:entropy-regularity},
and Corollary~\ref{cor:strong-almost} gives finite entropy at every
intermediate time.  In particular, all intermediate marginals are
absolutely continuous.

Steps 1--6 of \cite[Proof of Theorem~1.1]{RajalaSturm} localize the plan
in a ball of radius $\ell>0$ and select two branches.  Fix this localization,
so that all selected geodesics have length at most $2\ell$.
The construction gives a fixed interval $[T,S]\subset(0,1)$ such that, for
every sufficiently small $\varepsilon>0$, one can find
$t=t_\varepsilon\in[T,S]$ and two optimal dynamical subplans of the same mass
$w_\varepsilon>0$.  Their normalized marginals agree at time $t$ and are
mutually singular at time $t+\varepsilon$.  Their equal mixture is also
optimal: the endpoint couplings of both branches are concentrated on the
support of the same optimal coupling.

The density bound in \cite[Step~4]{RajalaSturm} gives a constant $C_1$
that bounds both endpoint densities before normalization.  After division
by $w_\varepsilon$, their entropies at times $0$ and $1$ are therefore at
most $L_\varepsilon:=\log(C_1/w_\varepsilon)$.
The entropy estimate in \cite[(3.8)]{RajalaSturm} gives one branch
whose unnormalized time-$(t+\varepsilon)$ density $\rho$ satisfies
$\int\rho\log\rho\d\mm\ge w_\varepsilon\log(\varepsilon/C_2)$.
Dividing by $w_\varepsilon$ and
subtracting $\log w_\varepsilon$ gives
$D_\varepsilon\ge\log(\varepsilon/(C_2w_\varepsilon))$, where
$D_\varepsilon$ is the entropy of this branch's normalized marginal.
The constants $T,S,C_1$ depend on the fixed localized plan, and $C_2$ depends on
the measure of its localization ball.  None depends on $\varepsilon$
or $w_\varepsilon$.
In particular,
\begin{equation}
	\label{eq:branch-entropy-lower}
	D_\varepsilon-L_\varepsilon
	\ge\log\varepsilon-\log(C_1C_2).
\end{equation}
The endpoint density bounds and Lemma~\ref{lem:entropy-regularity} give
finite entropy at both endpoints of the normalized branches and their
equal mixture.  Corollary~\ref{cor:strong-almost} on $[0,1]$ then gives
finite entropy at every intermediate time.  Thus it also applies to all
three time restrictions used below.

\textbf{Step 2. The three entropy inequalities.}
Let $A_\varepsilon$ be the entropy of the common normalized marginal at time
$t$, and let $B_\varepsilon,D_\varepsilon$ be the entropies of the two
normalized marginals at time $t+\varepsilon$.  Since the latter measures are
mutually singular, their equal mixture has entropy
$(B_\varepsilon+D_\varepsilon)/2-\log2$.

Apply Corollary~\ref{cor:strong-almost} first to the equal mixture on
$[0,t+\varepsilon]$, evaluated at $t$.  Apply it twice more, once to each
branch on $[t,1]$, evaluated at $t+\varepsilon$.  These three applications
give
\begin{align*}
	A_\varepsilon
	&\le \frac{\varepsilon}{t+\varepsilon}L_\varepsilon
	+\frac{t}{t+\varepsilon}
	\left(\frac{B_\varepsilon+D_\varepsilon}{2}-\log2\right)
	+\delta+r_0,\\
	B_\varepsilon
	&\le \frac{\varepsilon}{1-t}L_\varepsilon
	+\frac{1-t-\varepsilon}{1-t}A_\varepsilon
	+\delta+r_\uparrow,\\
	D_\varepsilon
	&\le \frac{\varepsilon}{1-t}L_\varepsilon
	+\frac{1-t-\varepsilon}{1-t}A_\varepsilon
	+\delta+r_\downarrow.
\end{align*}
Here $r_0,r_\uparrow,r_\downarrow$ are the curvature terms from
\eqref{eq:almost-strong-cd}.  For $\delta=0$, these are the inequalities in
\cite[Step~7]{RajalaSturm}.

For each normalized subplan, the Wasserstein distance between its
marginals at times $a$ and $b$ is at most $2\ell|b-a|$.
The factors in \eqref{eq:almost-strong-cd} therefore give
\[
 |r_0|\le2|K|\ell^2t\varepsilon,
 \qquad
 |r_\uparrow|,|r_\downarrow|
 \le2|K|\ell^2\varepsilon(1-t-\varepsilon).
\]
Substitute the last two entropy inequalities into the first and solve
for $A_\varepsilon$.  Inserting this bound into the inequality for
$D_\varepsilon$ gives
\begin{equation}
	\label{eq:branching-main-estimate}
	\begin{split}
		\varepsilon(D_\varepsilon-L_\varepsilon)
		\le{}&-t(1-t-\varepsilon)\log2\\
		&+\bigl[(1-t-\varepsilon)(2t+\varepsilon)+\varepsilon\bigr]\delta
		+R_{\varepsilon,\ell}.
	\end{split}
\end{equation}
Here the curvature remainder is
\[
\begin{aligned}
 R_{\varepsilon,\ell}
 &=(t+\varepsilon)(1-t-\varepsilon)r_0
   +\frac{t(1-t-\varepsilon)}2(r_\uparrow+r_\downarrow)
   +\varepsilon r_\downarrow,\\
 |R_{\varepsilon,\ell}|
 &\le2|K|\ell^2\varepsilon(t+\varepsilon)(1-t-\varepsilon)
 \longrightarrow0
 \qquad(\varepsilon\downarrow0).
\end{aligned}
\]

\textbf{Step 3. The contradiction.}
Choose $\varepsilon_j\downarrow0$.  After passing to a subsequence, assume
$t_{\varepsilon_j}\to t_\infty\in[T,S]$.  Since
$\varepsilon\log\varepsilon\to0$, the lower bound
\eqref{eq:branch-entropy-lower} gives
$\liminf_{j\to\infty}\varepsilon_j
(D_{\varepsilon_j}-L_{\varepsilon_j})\ge0$.
In \eqref{eq:branching-main-estimate}, the coefficient of $\delta$ tends to
$2t_\infty(1-t_\infty)$.  Taking the upper limit gives
\[
0\le-t_\infty(1-t_\infty)(\log2-2\delta).
\]
Since $t_\infty(1-t_\infty)\ge T(1-S)>0$, this contradicts
$\delta<\frac12\log2$.
\end{proof}
\begin{remark}
	Example~\ref{ex:tripod} illustrates the limitations of the Rajala--Sturm
	argument when $\delta$ is sufficiently large.  Establishing essential
	non-branching in this range may require entropy inequalities for more
	than two marginals.
\end{remark}
\subsection{Finite-dimensional BCD and essential non-branching}

\begin{proposition}[Passage to infinite dimension]\label{prop:finite-to-infinite}
For $1<N<\infty$, $\BCD(K,N)$ implies $\BCD_0(K,\infty)$.
In particular, every $\BCD(K,N)$ space with $1<N<\infty$ is essentially
non-branching.
\end{proposition}
\begin{proof}
The analogous implication for $(K,N)$-convex functions is proved in
\cite[Lemma~2.12]{EKS}.

Fix $\Omega=\sum_i\lambda_i\delta_{\mu_i}$ with
$\mu_i\in\mathrm D(\Ent_\mm)$, and
let $\bar\mu$ be a barycenter given by $\BCD(K,N)$.  Set
$E_i=\Ent_\mm(\mu_i)$, $\bar E=\Ent_\mm(\bar\mu)$,
$r_i=W_2(\bar\mu,\mu_i)$, and $q_i=Kr_i^2/N$.  

For $q<\pi^2$, define
\[
 A(q)=
 \begin{cases}
  \sqrt q/\sin\sqrt q,&q>0,\\
  1,&q=0,\\
  \sqrt{-q}/\sinh\sqrt{-q},&q<0,
 \end{cases}
 \qquad
 B(q)=
 \begin{cases}
  \sqrt q\cot\sqrt q,&q>0,\\
  1,&q=0,\\
  \sqrt{-q}\coth\sqrt{-q},&q<0.
 \end{cases}
\]
Thus $A(q_i)=r_i/s_{K/N}(r_i)$ and
$B(q_i)=r_i/t_{K/N}(r_i)$.

The proof uses the estimates $B(q)\le 1-q/3$ and
$\log A(q)\ge q/6$.  In each estimate, the two sides have the same value
and derivative at $q=0$.
The sum $1/3+1/6=1/2$ gives the coefficient
$K/2$ in the entropy inequality.  For the first estimate, let $0<x<\pi$.
The function $g_+(x)=(1-x^2/3)\sin x-x\cos x$ is nonnegative because
$g_+(0)=0$ and $g_+'(x)=x(\sin x-x\cos x)/3\ge0$; here
$(\sin x-x\cos x)'=x\sin x\ge0$.  Hence
$x\cot x\le1-x^2/3$.  For $x>0$, the same argument applied to
$g_-(x)=(1+x^2/3)\sinh x-x\cosh x$ gives
$x\coth x\le1+x^2/3$.  This proves the estimate for $B$.  Moreover,
\[
 \frac{\mathrm d}{\mathrm d x}\log\frac{x}{\sin x}
 =\frac{1-x\cot x}{x}\ge\frac{x}{3},
 \qquad
 \frac{\mathrm d}{\mathrm d x}\log\frac{x}{\sinh x}
 =\frac{1-x\coth x}{x}\ge-\frac{x}{3}.
\]
Integration from $0$ to $x$ proves the estimate for $\log A$.

If $\bar E=-\infty$, there is nothing to prove.  Otherwise,
\eqref{eq:finite-bcd} and the positivity of its left-hand side imply that
$\bar E\in\R$ and $\sum_i\lambda_iB(q_i)>0$.  The BCD inequality becomes
\[
 \sum_i\lambda_i A(q_i)e^{-E_i/N}
 \le e^{-\bar E/N}\sum_i\lambda_iB(q_i).
\]
Taking logarithms gives
\begin{equation}\label{eq:finite-to-infinite-log}
 \bar E\le
 -N\log\left(\sum_i\lambda_iA(q_i)e^{-E_i/N}\right)
 +N\log\left(\sum_i\lambda_iB(q_i)\right).
\end{equation}
Hence, by Jensen's inequality for the concave function $\log$,
\[
 \begin{aligned}
  \log\left(\sum_i\lambda_i A(q_i)e^{-E_i/N}\right)
  &\ge \sum_i\lambda_i
  \left(\log A(q_i)-\frac{E_i}{N}\right)\\
  &\ge \frac16\sum_i\lambda_iq_i
  -\frac1N\sum_i\lambda_iE_i.
 \end{aligned}
\]

Since $\sum_i\lambda_iB(q_i)>0$, the estimate for $B$ gives
\[
 \log\left(\sum_i\lambda_iB(q_i)\right)
 \le \log\left(1-\frac13\sum_i\lambda_iq_i\right)
 \le-\frac13\sum_i\lambda_iq_i.
\]

Substitution in \eqref{eq:finite-to-infinite-log} yields
\[
 \bar E\le \sum_i\lambda_iE_i
 -\frac{N}{2}\sum_i\lambda_iq_i
 =\sum_i\lambda_iE_i-\frac K2\sum_i\lambda_i r_i^2.
\]
Since $\sum_i\lambda_i r_i^2=\Var(\Omega)$, this is the defining
inequality of $\BCD_0(K,\infty)$.
Essential non-branching follows from Theorem~\ref{thm:enb} with $\delta=0$.
\end{proof}

\subsection{Stability of the $\BCD_\delta$ condition}
We prove stability under pointed measured Gromov--Hausdorff convergence.
 For compact spaces with probability reference measures, the
theorem also gives stability under measured Gromov--Hausdorff convergence.

We first bound the supports of barycenters.  If
$\supp\mu_i\subset B_R(o)$ for every $i$, then every barycenter of
$\sum_i\lambda_i\delta_{\mu_i}$ is supported in $\overline B_{2R}(o)$.
Indeed, for $\nu\in\Bary(\sum_i\lambda_i\delta_{\mu_i})$, define
\[
P(y)=
\begin{cases}
 y,& \dist(y,o)\le 2R,\\
 o,& \dist(y,o)>2R,
\end{cases}.
\]
The measure $P_\#\nu$ is a competitor for $\nu$.
If $\pi_i\in\Pi(\nu,\mu_i)$ is optimal, $\dist(y,o)>2R$, and
$x\in\supp\mu_i$, then
$\dist(o,x)\le R<\dist(y,o)-R\le\dist(y,x)$.
Thus $(P,\operatorname{id})_\#\pi_i$ has strictly smaller cost on a set
of positive mass whenever $\nu(X\setminus\overline B_{2R}(o))>0$.
Summing over $i$ contradicts the minimality of $\nu$.

\begin{theorem}[Stability]\label{thm:pointed-stability}
Let $(X_j,\dist_j,\mm_j,o_j)$ and $(X,\dist,\mm,o)$ be pointed proper
metric measure spaces.  Suppose that $(X_j,\dist_j,\mm_j,o_j)$ converges to
$(X,\dist,\mm,o)$ in the pointed measured Gromov--Hausdorff topology.
If $K_j\to K$,
$\limsup_j\delta_j\le\delta$, and each $X_j$ satisfies
$\BCD_{\delta_j}(K_j,\infty)$, then $X$ satisfies
$\BCD_\delta(K,\infty)$.
\end{theorem}

\begin{proof}
We isometrically embed $X_j$ and $X$ into a common proper metric space $Z$
so that $o_j\to o$, $X_j\to X$ locally in Hausdorff distance, and
$\mm_j\rightharpoonup\mm$ locally weakly; see
\cite[Section~3.5, proof of Proposition~3.30]{GigliMondinoSavare}
for the common-embedding construction.  
We identify the spaces and measures with their images under these embeddings.
Here local weak convergence means convergence against every function in
$C_c(Z)$.  Write $W_2^Z$ for the quadratic Wasserstein distance induced by
the metric of $Z$.

First let $\Omega=\sum_i\lambda_i\delta_{\mu_i}$, where the
$\mu_i\in\mathrm D(\Ent_\mm)$ are compactly supported.
Choose a compact ball $B\subset Z$ containing their supports in its
interior, with $\mm(\partial B)=0$.
Then $\mm_j(B)\to\mm(B)>0$, and the normalized restrictions
$\mm_j|_B/\mm_j(B)$ converge to $\mm|_B/\mm(B)$ in $W_2^Z$.
Apply the entropy convergence result
\cite[Proposition~4.7]{GigliMondinoSavare} on $B$.
Normalization adds $\log\mm_j(B)$ and $\log\mm(B)$ to the respective
entropies.  Thus, for all sufficiently large $j$, there are
$\mu_{j,i}\in\mathrm D(\Ent_{\mm_j})$, supported in $B$, such that
	\[
	W_2^Z(\mu_{j,i},\mu_i)\to0,
	\qquad
	\Ent_{\mathfrak m_j}(\mu_{j,i})
	\to\Ent_{\mathfrak m}(\mu_i).
	\]
	Let $\bar\mu_j$ be a barycenter of
	$\Omega_j:=\sum_i\lambda_i\delta_{\mu_{j,i}}$ satisfying the
	$\BCD_{\delta_j}(K_j,\infty)$ inequality.
	
	By the support bound above, the measures $\bar\mu_j$ are supported
	in a fixed compact subset of $Z$.  Hence, after passing to a subsequence,
	$W_2^Z(\bar\mu_j,\bar\mu)\to0$ for some $\bar\mu\in\cP(X)$.
To check minimality, fix $\eta\in\Pp(X)$ and choose compactly supported
$\eta^R$ with $W_2(\eta^R,\eta)\to0$ as $R\to\infty$.
For each fixed $R$, local Hausdorff convergence gives probability measures
$\eta_j^R$ on $X_j$, supported in a common compact subset of $Z$, such that
$W_2^Z(\eta_j^R,\eta^R)\to0$ as $j\to\infty$.
For fixed $R$, all measures in this comparison are supported in a common
compact set, where narrow convergence is equivalent to $W_2^Z$-convergence.
Passing to the limit in the barycenter inequality for $\bar\mu_j$ gives
\[
 \sum_i\lambda_iW_2^2(\bar\mu,\mu_i)
 \le\sum_i\lambda_iW_2^2(\eta^R,\mu_i).
\]
Letting $R\to\infty$ proves $\bar\mu\in\Bary(\Omega)$.
The convergence of the individual costs also gives
$\Var(\Omega_j)\to\Var(\Omega)$.
For the barycenters, restrict the reference measures to a larger compact
ball containing all the supports, again with $\mm$-null boundary, and
normalize.  The lower semicontinuity statement in
\cite[Proposition~4.7]{GigliMondinoSavare} and the almost BCD inequalities
on $X_j$ then give
	\[
	\Ent_{\mathfrak m}(\bar\mu)
	\le
	\sum_i\lambda_i\Ent_{\mathfrak m}(\mu_i)
	-\frac K2\Var(\Omega)+\delta.
	\]
	Thus the almost BCD inequality holds for compactly supported marginals.
	
	The proof of Lemma~\ref{lem:entropy-regularity} uses only compactly supported marginals.
	Applying the same argument on $X$ therefore gives Gaussian volume
	growth and, consequently, the $W_2$-lower semicontinuity of
	$\Ent_{\mathfrak m}$.
	
	Now let $\mu_i=\rho_i\mm\in\mathrm D(\Ent_\mm)$ be arbitrary.
	For all sufficiently large $R$, we have $\mu_i(B_R(o))>0$ for every $i$.
	For these $R$, let $\mu_i^R$ be the normalized restriction of $\mu_i$
	to $B_R(o)$.
	Finite second moments give $W_2(\mu_i^R,\mu_i)\to0$.
	Since $\rho_i\log\rho_i\in L^1(\mm)$,
	\[
	\Ent_\mm(\mu_i^R)
	=\frac{1}{\mu_i(B_R(o))}
	\int_{B_R(o)}\rho_i\log\rho_i\d\mm
	-\log\mu_i(B_R(o))
	\longrightarrow\Ent_\mm(\mu_i).
	\]
	For $\Omega^R:=\sum_i\lambda_i\delta_{\mu_i^R}$,
	choose a barycenter $\bar\mu^R$ satisfying the almost BCD inequality.
	The coupling $\sum_i\lambda_i\delta_{(\mu_i^R,\mu_i)}$ has cost
	$\sum_i\lambda_iW_2^2(\mu_i^R,\mu_i)\to0$, so
	$\Omega^R\to\Omega$ in the quadratic Wasserstein distance on
	$\Pp(\Pp(X))$, with base distance $W_2$.
	Since $X$ is proper and geodesic, the consistency theorem
	\cite[Theorem~3]{LeGouicLoubes} applies and gives
	$\bar\mu^R\to\bar\mu\in\Bary(\Omega)$ in $W_2$
	along a subsequence.
	We also have $\Var(\Omega^R)\to\Var(\Omega)$. Passing to the limit by the
	lower semicontinuity of entropy proves the desired inequality for
	$\Omega$.
\end{proof}

\section{Riemannian rigidity and parallelogram estimates}\label{sec:exact}

\subsection{Tangent spaces and transport along geodesic lines}\label{sec:tangent-lines}

In this subsection, assume that $X$ is locally doubling and supports a
local weak $(1,1)$-Poincar\'e inequality.
Cheeger's differentiability theorem gives differentiability charts and a
finite-dimensional cotangent bundle with its canonical norm
\cite[Theorem~4.38 and Definition~4.42]{CheegerDifferentiability}.

By local finiteness and separability, choose a countable cover of $X$ by
open balls $B$ with $\mm(\partial B)=0$.
Every real-valued Lipschitz function on $\overline B$ extends to one on
$X$, and the differentiability charts restrict to the open set $B$.
Since $\mm(\partial B)=0$, the space
$(\overline B,\dist,\mm|_{\overline B})$ is a differentiability space in
the sense of \cite[Section~2.1]{CKS}.
At every interior point, restriction to $\overline B$ preserves
the pointed measured blow-ups and the cotangent norm.
We apply the results of Cheeger--Kleiner--Schioppa
\cite[Section~7]{CKS} to these differentiability spaces.

Fix an interior point $x$ at which these results apply.
Given scales $r_j\downarrow0$,
pass to a subsequence and choose normalization constants $c_j>0$ as in
\cite[Remarks~7.7 and~7.10]{CKS}.  We obtain a pointed measured tangent
$(Y,\dist_Y,\mh,o)$ and a blow-up $\widehat\phi$ of a differentiability chart, with
\[
(X,r_j^{-1}\dist,c_j\mm,x)\to(Y,\dist_Y,\mh,o)
\]
in the pointed measured Gromov--Hausdorff sense.
Set $V=T_xX$ and equip it with the norm
$\|\cdot\|_x$ dual to Cheeger's cotangent norm
\cite[pp.~460--461]{CheegerDifferentiability}.
By \cite[Corollary~7.99(1)]{CKS}, this blow-up map is a metric
submersion
\begin{equation}\label{eq:tangent-submersion}
\widehat\phi:Y\longrightarrow(V,\|\cdot\|_x),
\qquad \widehat\phi(o)=0.
\end{equation}
Since $V$ is finite-dimensional, identify $V^*$ with $T_x^*X$ and denote by
$\|\cdot\|_{x,*}$ the norm on $V^*$ dual to $\|\cdot\|_x$; thus
$\|r\|_{x,*}=\sup_{\|v\|_x\le1}|r(v)|$.  For every Lipschitz function $f$,
the Cheeger differential satisfies
\begin{equation}\label{eq:cks-sobolev-norm}
 \|\mathrm D_x f\|_{x,*}=\operatorname{lip}f(x)=|\mathrm Df|(x)
 \qquad\text{for $\mm$-almost every $x$}.
\end{equation}
Here $\operatorname{lip}f$ denotes the common value of the upper and lower
pointwise Lipschitz constants of $f$, which agree $\mm$-almost everywhere,
and $|\mathrm Df|$ is the minimal $2$-weak upper gradient, understood
locally.  The identities
relating the Cheeger differential to the pointwise Lipschitz constants follow from
\cite[Corollary~4.41, Theorem~6.1]{CheegerDifferentiability}.
The identification with the minimal $2$-weak upper gradient is given by
\cite[Theorem~6.2]{AGSCalculus}, applied locally.

An Alberti representation expresses a measure as an integral of
one-dimensional measures carried by Lipschitz curves.  This construction
originates in \cite{AlbertiRankOne}; for its formulation on metric spaces,
see \cite[Section~2]{BateStructure} and \cite{SchioppaAlberti}.
We use the definition allowing curves with unbounded domains
\cite[Definition~2.12]{CKS}.
For each $v\in V$ with $\|v\|_x=1$, Cheeger--Kleiner--Schioppa
\cite[Theorem~7.18]{CKS} provide a Borel family of complete unit-speed
geodesic lines $\gamma:\R\to Y$, denoted by
$\operatorname{Lines}(\widehat\phi,v)$, satisfying
\begin{equation}\label{eq:tangent-linear-line}
\widehat\phi(\gamma(t))=\widehat\phi(\gamma(0))+tv.
\end{equation}
The measure associated with each line is its arclength measure.  More precisely,
\[
 \mh=\int_{\operatorname{Lines}(\widehat\phi,v)}
 \gamma_\#\cL^1\d P_v(\gamma)
\]
for a Borel measure $P_v$ provided by the same theorem.
For every $r\in V^*$, write
$\widehat r:=r\circ\widehat\phi:Y\to\R$.
Kantorovich potentials on $Y$ are taken for the cost $\dist_Y^2/2$.
For a function $\psi:Y\to\R$, its $c$-transform is
$\psi^c(z):=\inf_{y\in Y}\{\dist_Y^2(y,z)/2-\psi(y)\}$.

\begin{lemma}[Entropy estimate along geodesic lines]\label{lem:tangent-line-entropy}
Let $\nu=\rho\mh$ be a compactly supported probability measure with finite
entropy.  Fix
$r\in V^*$ and $\alpha\in\R$ with $\alpha\|r\|_{x,*}^2<1$.
Then there exist a compactly supported
probability measure $\mu_{\alpha,r}$ with finite entropy and an optimal
coupling $\pi_{\alpha,r}$ from $\nu$ to $\mu_{\alpha,r}$ such that
$\alpha\widehat r^{\,2}/2$ is a Kantorovich potential for
$\pi_{\alpha,r}$ and
\begin{equation}\label{eq:tangent-line-entropy}
 \Ent_{\mh}(\mu_{\alpha,r})
 \le\Ent_{\mh}(\nu)-\log(1-\alpha\|r\|_{x,*}^2).
\end{equation}
\end{lemma}

\begin{proof}
If $r=0$, take $\mu_{\alpha,0}=\nu$ and the diagonal coupling.  Henceforth
fix $r\ne0$.  

Since $V$ is finite-dimensional, the supremum in the
definition of $\|r\|_{x,*}$ is attained.  Choose $v_r\in V$ such that
$\|v_r\|_x=1$ and $r(v_r)=\|r\|_{x,*}$, and define
$q_r:=\widehat r/\|r\|_{x,*}:Y\to\R$.
For $y,z\in Y$, the fact that $\widehat\phi$ is $1$-Lipschitz gives
$|q_r(y)-q_r(z)|\le \|\widehat\phi(y)-\widehat\phi(z)\|_x
\le \dist_Y(y,z)$.
Thus $q_r$ is $1$-Lipschitz as a function on $Y$.

Apply the Alberti representation above with $v=v_r$, and write
$\mathscr L:=\operatorname{Lines}(\widehat\phi,v_r)$ and $P:=P_{v_r}$.
For every $\gamma\in\mathscr L$, \eqref{eq:tangent-linear-line} gives
$q_r(\gamma(t))=t+b_\gamma$, where $b_\gamma:=q_r(\gamma(0))$.
Let $\widetilde{\mathfrak m}:=P\otimes\cL^1$ on
$\mathscr L\times\R$.  For $t\in\R$, let
$e_t:\mathscr L\to Y$ be the evaluation map $e_t(\gamma):=\gamma(t)$.
Write $\mathrm{ev}:\mathscr L\times\R\to Y$ for the joint evaluation map
$\mathrm{ev}(\gamma,t):=e_t(\gamma)$.
With the Borel structure on parametrized curves used in
\cite[Definition~2.12]{CKS}, the map $\mathrm{ev}$ is Borel.  Since $q_r$
is continuous, $\gamma\mapsto b_\gamma=q_r(e_0(\gamma))$ is also Borel.
The defining identity for the Alberti representation is
$\mathrm{ev}_\#\widetilde{\mathfrak m}=\mh$.

Put $\beta=\alpha\|r\|_{x,*}^2$ and $J=1-\beta>0$.  Lift
\(\nu=\rho\mh\) to $\mathscr L\times\R$ by
$\widetilde\nu:=(\rho\circ\mathrm{ev})\widetilde{\mathfrak m}$.
Then $\mathrm{ev}_\#\widetilde\nu=\nu$ and
$\Ent_{\widetilde{\mathfrak m}}(\widetilde\nu)=\Ent_{\mh}(\nu)$.

On each line, change the parameter by the affine map
\(S_\beta(\gamma,t)=(\gamma,Jt-\beta b_\gamma)\).  Its derivative in
the line parameter is \(J\), so
\((S_\beta)_\#\widetilde{\mathfrak m}
=J^{-1}\widetilde{\mathfrak m}\).  The one-dimensional change-of-variables
formula gives
\[
 \Ent_{\widetilde{\mathfrak m}}
 \bigl((S_\beta)_\#\widetilde\nu\bigr)
 =\Ent_{\mh}(\nu)-\log J.
\]
Define
\[
 \mu_{\alpha,r}:=\mathrm{ev}_\#(S_\beta)_\#\widetilde\nu,
 \qquad
 \pi_{\alpha,r}:=(\mathrm{ev},\mathrm{ev}\circ S_\beta)_\#\widetilde\nu.
\]
The measure $\pi_{\alpha,r}$ is a coupling of \(\nu\) and \(\mu_{\alpha,r}\).
For $\pi_{\alpha,r}$-almost every $(y,z)$, the change of parameter along a
unit-speed line gives $\dist_Y(y,z)=|\beta q_r(y)|$.
Since $q_r$ is bounded on $\supp\nu$, the second marginal is supported
in a bounded neighborhood of $\supp\nu$.  The tangent $Y$ is proper,
so $\mu_{\alpha,r}$ is compactly supported.
Relative entropy does not increase when both measures are pushed forward
by a Borel map \cite[(7.4)]{AGSCalculus}; see also \cite{Csiszar}.
Here we may restrict $\widetilde{\mathfrak m}$ to
$\mathrm{ev}^{-1}(\supp\mu_{\alpha,r})$, whose measure is finite because
$\mathrm{ev}_\#\widetilde{\mathfrak m}=\mh$ and $\mh$ is locally finite.
Applying the inequality to $\mathrm{ev}$ gives
\[
 \Ent_{\mh}(\mu_{\alpha,r})
 \le \Ent_{\widetilde{\mathfrak m}}
 \bigl((S_\beta)_\#\widetilde\nu\bigr).
\]
Together with the change-of-variables identity, this proves
\eqref{eq:tangent-line-entropy}.  Compact support and local finiteness of
$\mh$ bound this entropy from below, so it is finite.

It remains to verify optimality.  For
$\widetilde\nu$-almost every $(\gamma,t)\in\mathscr L\times\R$, set
\(y=\gamma(t)\) and \(z=\gamma(Jt-\beta b_\gamma)\).
Put $a=q_r(y)=t+b_\gamma$.  Then $q_r(z)=Ja$ and
$\dist_Y(y,z)=|\beta a|$.  For any $w\in Y$, the $1$-Lipschitz property of
$q_r$ gives $\dist_Y(w,z)\ge|q_r(w)-Ja|$.  The quadratic
function $s\mapsto (s-Ja)^2/2-\beta s^2/2$
has its minimum at $s=a$, with value $-\beta Ja^2/2$.
The point $w=y$ attains this value, so
$(\beta q_r^2/2)^c(z)=-\beta Ja^2/2$.  Hence
\[
 \frac\beta2q_r(y)^2+
 \left(\frac\beta2q_r^2\right)^c(z)
 =\frac12\dist_Y^2(y,z).
\]
Thus $\alpha\widehat r^{\,2}/2=\beta q_r^2/2$ is a Kantorovich potential for
$\pi_{\alpha,r}$, and the coupling is optimal.
\end{proof}

\subsection{Quantitative parallelogram estimate}

\begin{theorem}[Quantitative infinitesimal Hilbertianity]
\label{thm:quantitative-hilbertianity}
Suppose that $X$ is locally doubling, supports a local weak
$(1,1)$-Poincar\'e inequality, and satisfies
$\BCD_\delta(K,\infty)$.  Then
\begin{equation}\label{eq:fiber-pstar}
 \mathfrak p_{\mathrm{par}}(T_x^*X)\le8\sqrt\delta
 \qquad\text{for $\mm$-almost every }x.
\end{equation}
\end{theorem}

\begin{proof}
Fix a point $x$ at which Cheeger's differentiability theorem and the tangent
theorem of Cheeger--Kleiner--Schioppa apply.
Choose a pointed measured tangent and a blow-up map as in
\eqref{eq:tangent-submersion}--\eqref{eq:tangent-linear-line}.
If the tangent is obtained from scales $r_j\downarrow0$, then the rescaled
spaces satisfy $\BCD_\delta(r_j^2K,\infty)$.  Multiplying the reference
measure by a positive constant does not change this condition.
Since $X$ is complete, geodesic, and locally doubling, it is proper.
Theorem~\ref{thm:pointed-stability} therefore applies to the rescaled
spaces and shows that
$(Y,\dist_Y,\mh)$ satisfies $\BCD_\delta(0,\infty)$.

Fix $p,q\in T_x^*X$ with $(p,q)\ne(0,0)$ and set
\[
 a=\|p+q\|_{x,*}^2,\quad b=\|p-q\|_{x,*}^2,\quad
 c=2\|p\|_{x,*}^2,\quad d=2\|q\|_{x,*}^2,\quad
 M=\max\{a,b,c,d\}.
\]
Let $\nu=\rho\mh$ be a compactly supported probability measure with finite
entropy.  For
$s\in\{-1,1\}$ and $0<\varepsilon M\le1/2$, take
\[
 \psi_1=\frac{s\varepsilon}{2}\widehat{p+q}^{\,2},\quad
 \psi_2=\frac{s\varepsilon}{2}\widehat{p-q}^{\,2},\quad
 \psi_3=-s\varepsilon\widehat p^{\,2},\quad
 \psi_4=-s\varepsilon\widehat q^{\,2}.
\]
Their sum is zero, and each has at most quadratic growth.
Apply Lemma~\ref{lem:tangent-line-entropy} separately to the four
potentials, with $\alpha=s\varepsilon$ for the first two and
$\alpha=-2s\varepsilon$ for the last two.
It gives four probability measures $\mu_i$ with finite entropy
and optimal couplings for which the functions $\psi_i$ are Kantorovich
potentials.  

For any
$\eta\in\Pp(Y)$, Kantorovich duality \cite[Theorem~5.10]{Villani2009} gives
\[
 \frac12W_2^2(\eta,\mu_i)
 \ge \int_Y\psi_i\d\eta+\int_Y\psi_i^c\d\mu_i.
\]
After summing these inequalities with weight $1/4$, the terms containing
$\eta$ cancel because $\sum_i\psi_i=0$.  At $\eta=\nu$, equality holds in
every inequality by the choice of the optimal couplings.  Hence
$\nu\in\Bary(\frac14\sum_{i=1}^4\delta_{\mu_i})$.

Apply Proposition~\ref{prop:every-barycenter} at this barycenter.
Combining its inequality with the four entropy bounds from
Lemma~\ref{lem:tangent-line-entropy} gives
\[
 \log(1-s\varepsilon a)+\log(1-s\varepsilon b)
 +\log(1+s\varepsilon c)+\log(1+s\varepsilon d)\le4\delta.
\]

Set $G=a+b-c-d$.  Choose $s$ so that $s(c+d-a-b)=|G|$.  Since
$\log(1+z)\ge z-z^2$ for $|z|\le1/2$, the left-hand side is at least
$\varepsilon|G|-4\varepsilon^2M^2$.

If $G=0$, there is nothing to prove.  Otherwise, the choice
$\varepsilon=|G|/(8M^2)$ satisfies the required bound, since $|G|\le2M$ gives
$\varepsilon M\le1/4$.  Substitution yields
$|G|^2/(16M^2)\le4\delta$, hence $|G|/M\le8\sqrt\delta$.
Taking the supremum
over $p,q$ proves
\eqref{eq:fiber-pstar}.
\end{proof}

When $\delta=0$, Theorem~\ref{thm:quantitative-hilbertianity} gives the
parallelogram identity in $T_x^*X$.  In view of
\eqref{eq:cks-sobolev-norm}, Hilbertianity of the cotangent norm
$\mm$-almost everywhere is equivalent to infinitesimal Hilbertianity; see also
\cite[Corollary~6.7]{ErikssonBiqueSoultanis}.

Write $W^{1,2}(X,\dist,\mm)$ for the space of $L^2(\mm)$ functions whose
minimal $2$-weak upper gradient belongs to $L^2(\mm)$
\cite[Theorem~6.2]{AGSCalculus}.  The Cheeger energy is
normalized by $\Ch(f):=\frac12\int_X|\mathrm Df|^2\d\mm$.

For a space with the local doubling and Poincar\'e properties above, set
\[
 \mathfrak H(X):=\operatorname*{ess\,sup}_{x\in X}
 \mathfrak p_{\mathrm{par}}(T_x^*X),
 \qquad
 \Delta_{\BCD}^{K}(X):=
 \inf\{\varepsilon\ge0:X\text{ satisfies }
 \BCD_\varepsilon(K,\infty)\},
\]
where the essential supremum is taken with respect to $\mm$ and
$\inf\varnothing=+\infty$.

\begin{corollary}[Parallelogram estimate for the Cheeger energy]
	\label{cor:quantitative-cheeger}
	Under the hypotheses of Theorem~\ref{thm:quantitative-hilbertianity},
	\[
		\Delta_{\BCD}^{K}(X)\ge\frac1{64}\mathfrak H(X)^2.
	\]
	For every $f,g\in W^{1,2}(X,\dist,\mm)$, one also has
	\begin{align}
		&\left|\Ch(f+g)+\Ch(f-g)-2\Ch(f)-2\Ch(g)\right|\notag\\
		&\qquad\le16\sqrt\delta\bigl(\Ch(f)+\Ch(g)\bigr).
		\label{eq:cheeger-parallelogram-defect}
	\end{align}
\end{corollary}

\begin{proof}
	Take the essential supremum in \eqref{eq:fiber-pstar} and then the infimum
	over admissible errors.  For Lipschitz $f,g\in W^{1,2}(X,\dist,\mm)$, apply
	\eqref{eq:fiber-pstar} to $p=\mathrm D_xf$ and $q=\mathrm D_xg$.
	The denominator in the parallelogram defect is at most
	$2(|\mathrm Df|^2+|\mathrm Dg|^2)$.  Integration therefore gives
	\eqref{eq:cheeger-parallelogram-defect}.
	By \cite[Theorem~1.5]{BjornBjornLocalPI}, compactly supported Lipschitz
	functions are dense in $W^{1,2}(X,\dist,\mm)$ in the Sobolev norm
	under our assumptions.
	The estimate therefore extends to every $f,g\in W^{1,2}(X,\dist,\mm)$.
\end{proof}

\begin{proof}[Proof of Theorem~\ref{thm:intro-almost}]
Part~\textup{(i)} is Proposition~\ref{prop:every-barycenter},
part~\textup{(ii)} is Theorem~\ref{thm:enb}, and part~\textup{(iii)} is
Theorem~\ref{thm:quantitative-hilbertianity}.
\end{proof}

\begin{proof}[Proof of Theorem~\ref{thm:main-rigidity}]
The implication $\RCD(K,N)\Rightarrow\BCD(K,N)$ follows from
\cite[Theorem~1.5]{HanLiuZhuBCD}.  Conversely, assume $\BCD(K,N)$.
Proposition~\ref{prop:finite-to-infinite} gives $\BCD_0(K,\infty)$ and
essential non-branching.  Corollary~\ref{cor:strong-almost}, with $\delta=0$,
gives strong $\CD(K,\infty)$.
For $\mu_0,\mu_1\in\mathrm D(\Ent_\mm)$ and $0<t<1$, the barycenter of
$(1-t)\delta_{\mu_0}+t\delta_{\mu_1}$ is therefore the time-$t$ point
$\mu_t$ of their unique Wasserstein geodesic
\cite[Corollary~1.4]{RajalaSturm}.

Put $D=W_2(\mu_0,\mu_1)$ and $\kappa=K/N$.
The case $D=0$ is immediate, so assume $D>0$.
If $K>0$, the midpoint case of \eqref{eq:finite-bcd} forces
$c_\kappa(D/2)>0$, hence $D<\pi/\sqrt\kappa$.
The addition formula for $s_\kappa$ then rewrites the two-point case of
\eqref{eq:finite-bcd} as \eqref{eq:cde-definition}, proving strong
$\CDe(K,N)$.

Local doubling follows from \cite[Proposition~3.6]{EKS}.
Together with $\CD(K,\infty)$, it gives a local weak
$(1,1)$-Poincar\'e inequality by \cite[Theorem~1.2]{RajalaInterpolated};
see also \cite[Theorem~1.1]{RajalaPI}.
Corollary~\ref{cor:quantitative-cheeger}, with $\delta=0$, now gives
infinitesimal Hilbertianity.
Together with $\CDe(K,N)$, this implies $\RCD^*(K,N)$
\cite[Definition~3.16 and Theorem~3.17]{EKS}, and hence $\RCD(K,N)$
by \cite[Corollary~13.7]{CavallettiMilman} and
\cite[Corollary~1.2]{LiGlobalization}.
\end{proof}

\section{Finsler models and quantitative estimates}\label{sec:finsler}

\subsection{The log-determinant upper bound}

Let \(F\) be a Minkowski norm on \(\R^n\) and set
\(h(v):=F(v)^2/2\).
The $W_2$-convergence statements in this section use the Euclidean distance.
Throughout this section, assume
\begin{equation}\label{eq:ellipticity}
h\in C^3(\R^n\setminus\{0\}),
\qquad
aI\le g_v:=\mathrm D^2h(v)\le bI
\quad(v\ne0),
\end{equation}
for some \(0<a\le b<\infty\).  Since \(h\) is even and homogeneous of degree two,
\(g_{rv}=g_v\) for \(r>0\), and \(g_{-v}=g_v\).
These are the fundamental tensors of Minkowski geometry
\cite[Chapter~1]{BaoChernShen}.

The cost $h$ need not be twice differentiable at the origin.
We first mollify $h$ without increasing the log-determinant bound.
We then use compactness to pass the entropy inequality from the smooth
costs to $h$.

\subsubsection{Log-determinant defect and smoothing}

Let \(\Sym_n^+\) denote the cone of positive-definite symmetric
$n\times n$ matrices, and set
\(\mathcal G_F:=\{g_v:|v|=1\}\subset\Sym_n^+\).
The Minkowski determinant inequality states that
$A\mapsto(\det A)^{1/n}$ is concave on $\Sym_n^+$; see
\cite[Section~4.2, (4.21)]{BhatiaPositive}.  Since the logarithm
is increasing and concave, $A\mapsto\log\det A$ is concave as well.  It is
strictly concave: if $A_t=A+tH$ remains positive definite, then
\[
 \frac{\mathrm d^2}{\mathrm d t^2}\log\det A_t
 =-\tr\!\left((A_t^{-1/2}HA_t^{-1/2})^2\right)<0
 \qquad(H\ne0).
\]
Define \(c_F(F)\) by \eqref{eq:intro-cf}.  Concavity of $\log\det$ and
\eqref{eq:ellipticity} give \(0\le c_F(F)\le n\log(b/a)\).

Since $\mathcal G_F$ is compact, finitely supported probability measures are
weakly dense in $\cP(\mathcal G_F)$.  Both terms in the definition of
$c_F(F)$ are weakly continuous.  Hence, for every
$\sigma\in\cP(\mathcal G_F)$,
\begin{equation}\label{eq:cf-prob}
\log\det\left(\int_{\mathcal G_F}A\d\sigma(A)\right)
-\int_{\mathcal G_F}\log\det A\d\sigma(A)
\le c_F(F).
\end{equation}
The supremum of the left-hand side over these probability measures is
$c_F(F)$.

Let \(\vartheta\in C_c^\infty(\R^n)\) be a nonnegative even mollifier with
unit integral and zero first moment.  Put
\[
\vartheta_\varepsilon(z)=\varepsilon^{-n}\vartheta(z/\varepsilon),
\qquad
\widetilde h_\varepsilon=h*\vartheta_\varepsilon,
\qquad
h_\varepsilon=\widetilde h_\varepsilon-\widetilde h_\varepsilon(0).
\]

\begin{lemma}[Uniform approximation]\label{lem:cost-approx}
For every \(\varepsilon>0\),
$aI\le \mathrm D^2h_\varepsilon\le bI$ on $\R^n$.  Moreover, there are
constants $C_0,C_1>0$, independent of \(\varepsilon\), such that
\begin{equation}\label{eq:cost-grad-approx}
\|h_\varepsilon-h\|_{L^\infty(\R^n)}\le C_0\varepsilon^2,
\qquad
\|\mathrm D h_\varepsilon-\mathrm D h\|_{L^\infty(\R^n)}\le C_1\varepsilon.
\end{equation}
\end{lemma}

\begin{proof}
Homogeneity of degree two gives $h(v)=O(|v|^2)$ and $\mathrm D h(v)=O(|v|)$
near the origin, so $h\in C^1(\R^n)$ and $\mathrm D h(0)=0$.
Integrating \eqref{eq:ellipticity} along line segments, split at the origin
when necessary, shows that $\mathrm D h$ is $b$-Lipschitz and gives
\begin{equation}\label{eq:strong-conv-h}
\frac a2|z|^2
\le h(x+z)-h(x)-\mathrm D h(x)\cdot z
\le\frac b2|z|^2
\qquad(x,z\in\R^n).
\end{equation}
Since $\mathrm D h$ is Lipschitz, weak differentiation gives
$\mathrm D^2h_\varepsilon=(\mathrm D^2h)*\vartheta_\varepsilon$,
with $\mathrm D^2h$ interpreted $\cL^n$-almost everywhere.
The mollifier is nonnegative and has unit integral, so convolution preserves
the bounds in \eqref{eq:ellipticity}.  To prove
the first estimate in \eqref{eq:cost-grad-approx}, apply
\eqref{eq:strong-conv-h} with increment \(-u\), integrate against
\(\vartheta_\varepsilon(u)\), and subtract the corresponding estimate at the
origin.  The second estimate follows from the global \(b\)-Lipschitz
continuity of \(\mathrm D h\).
\end{proof}

\begin{lemma}[Stability under smoothing]	\label{lem:defect-smoothing}
For every integer $k\ge1$, the following inequality holds for all positive
weights $\lambda_1,\ldots,\lambda_k$ with $\sum_i\lambda_i=1$ and all
points $z_1,\ldots,z_k\in\R^n$:
\begin{equation}\label{eq:smoothed-defect}
\log\det\left(\sum_i\lambda_i\mathrm D^2h_\varepsilon(z_i)\right)
-\sum_i\lambda_i\log\det \mathrm D^2h_\varepsilon(z_i)
\le c_F(F).
\end{equation}
\end{lemma}

\begin{proof}
The value of \(\mathrm D^2h\) at the origin is irrelevant for convolution
and may be chosen in \(\mathcal G_F\).  For each \(z\), the matrix
\(\mathrm D^2h_\varepsilon(z)
=\int \mathrm D^2h(z-u)\vartheta_\varepsilon(u)\d u\)
is the barycenter of a probability measure \(\sigma_{\varepsilon,z}\)
supported on \(\mathcal G_F\).  Let
\(G_i=\mathrm D^2h_\varepsilon(z_i)\) and
\(\sigma=\sum_i\lambda_i\sigma_{\varepsilon,z_i}\).  Concavity of
\(\log\det\) gives
\(\log\det G_i\ge\int\log\det A\d\sigma_{\varepsilon,z_i}(A)\).
Thus the left-hand side of \eqref{eq:smoothed-defect} is no larger than the
left-hand side of \eqref{eq:cf-prob} for \(\sigma\).
\end{proof}

\subsubsection{Jacobian identities and the entropy inequality}

\paragraph{Barycenters of points.}

Fix positive weights $\lambda_1,\ldots,\lambda_k$ with
$\sum_i\lambda_i=1$.  For points $x_1,\ldots,x_k\in\R^n$ and an even,
\(a\)-strongly convex function \(H\in C^1(\R^n)\),
the function \(z\mapsto\sum_i\lambda_iH(z-x_i)\) is coercive and
\(a\)-strongly convex.  Its unique minimizer, denoted by
\(b_H(x_1,\ldots,x_k)\), is the $H$-barycenter of these points
\cite[Definition~2.1]{BrizziFrieseckeRied}.

\begin{lemma}[Stability of the barycenter map]\label{lem:pointbar-stability}
If $H\in C^1(\R^n)$ is even and $a$-strongly convex with a
$b$-Lipschitz gradient, then, for $x=(x_i),y=(y_i)\in(\R^n)^k$,
\begin{equation}\label{eq:pointbar-lip}
|b_H(x)-b_H(y)|
\le\frac ba\left(\sum_i\lambda_i|x_i-y_i|^2\right)^{1/2}.
\end{equation}
If \(H_1,H_2\in C^1(\R^n)\) are even and \(a\)-strongly convex and
\(\|\mathrm D H_1-\mathrm D H_2\|_\infty\le\eta\), then
\begin{equation}\label{eq:pointbar-cost}
\|b_{H_1}-b_{H_2}\|_\infty\le\frac\eta a.
\end{equation}
In particular, \(\|b_{h_\varepsilon}-b_h\|_\infty\le C\varepsilon\).
\end{lemma}

\begin{proof}
Let \(\Phi_x(z)=\sum_i\lambda_iH(z-x_i)\), and set \(p=b_H(x)\) and
\(q=b_H(y)\).  The first-order optimality conditions are
$\mathrm D\Phi_x(p)=\mathrm D\Phi_y(q)=0$.
Adding the inequalities expressing strong convexity of $\Phi_x$ at $p$ and $q$ gives
$a|p-q|^2\le
(\mathrm D\Phi_x(q)-\mathrm D\Phi_x(p))\cdot(q-p)$.
By the Cauchy--Schwarz inequality and the $b$-Lipschitz continuity of $\mathrm D H$,
\[
\begin{aligned}
a|p-q|
&\le |\mathrm D\Phi_x(q)-\mathrm D\Phi_y(q)|\\
&\le b\sum_i\lambda_i|x_i-y_i|
\le b\left(\sum_i\lambda_i|x_i-y_i|^2\right)^{1/2}.
\end{aligned}
\]
This proves \eqref{eq:pointbar-lip}.

For \eqref{eq:pointbar-cost}, fix $x=(x_i)$ and put
$p=b_{H_1}(x)$ and $q=b_{H_2}(x)$.
Apply the same inequality for strong convexity to
$z\mapsto\sum_i\lambda_iH_1(z-x_i)$.
The optimality conditions give
$\sum_i\lambda_i\mathrm D H_1(p-x_i)
=\sum_i\lambda_i\mathrm D H_2(q-x_i)=0$, so
\[
a|p-q|
\le\left|\sum_i\lambda_i
\bigl(\mathrm D H_1(q-x_i)-\mathrm D H_2(q-x_i)\bigr)\right|
\le\eta.
\]
Taking the supremum over $x$ proves \eqref{eq:pointbar-cost}.
The last assertion follows from Lemma~\ref{lem:cost-approx}.
\end{proof}

\paragraph{Differential and Jacobian identities.}

For a convex function \(H:\R^n\to\R\) satisfying
$|H(v)|\le C(1+|v|^2)$ for some $C>0$, consider the transport cost
\(c(x,y)=H(x-y)\).  For $\mu,\nu\in\Pp(\R^n)$, write
\(\mathcal T_H(\mu,\nu)
:=\inf_{\pi\in\Pi(\mu,\nu)}\int H(x-y)\d\pi(x,y)\).
For the weights $(\lambda_i)$ fixed above and $\mu_i\in\Pp(\R^n)$,
an $H$-Wasserstein barycenter of $(\mu_i)$ minimizes
$\nu\mapsto\sum_i\lambda_i\mathcal T_H(\nu,\mu_i)$ over
$\nu\in\Pp(\R^n)$.  For even,
strongly convex $H$, the corresponding barycenter map on $(\R^n)^k$ is $b_H$.

For the quadratic cost, the barycenter balance relation appears in
Agueh--Carlier \cite[Proposition~3.8 and Remark~3.9]{AguehCarlier}.
On Riemannian manifolds, Cordero-Erausquin--McCann--Schmuckenschl\"ager
established the differential and Jacobian formulas for optimal maps
\cite[Proposition~4.1 and Theorem~4.2]{CorderoErausquinMcCannSchmuckenschlager}.
Kim--Pass proved first- and second-order balance relations for
Wasserstein barycenters, with an inequality at second order
\cite[Theorem~4.4]{KimPass}.  For finitely many compactly supported
marginals on a complete Riemannian manifold, Ma proved a Hessian equality
assuming that one marginal is absolutely continuous
\cite[Theorem~3.13]{MaWassersteinBarycenter}.
The following lemma gives the corresponding formulas for the transport
costs $H(x-y)$ considered in \cite{BrizziFrieseckeRied}.

\begin{lemma}[Differential and Jacobian identities]\label{lem:jacobian-identities}
Let \(H\in C^2(\R^n)\) be even and satisfy
\(aI\le \mathrm D^2H\le bI\).  Let
\(\mu_i=\rho_i\cL^n\) be compactly supported probability measures with finite
entropy, and let $\bar\mu\in\Pp(\R^n)$ minimize
$\nu\mapsto\sum_i\lambda_i\mathcal T_H(\nu,\mu_i)$ over $\Pp(\R^n)$.
Then $\bar\mu$ is compactly supported and $\bar\mu\ll\cL^n$;
write $\bar\mu=\bar\rho\cL^n$.
The optimal couplings from $\bar\mu$ to $\mu_i$ are induced by maps $T_i$
with representatives that are approximately differentiable
$\bar\mu$-almost everywhere.  For
\(\bar\mu\)-almost every $x$, set $v_i=x-T_i(x)$,
$G_i=\mathrm D^2H(v_i)$, and $G=\sum_i\lambda_iG_i$.  Then the following
statements hold; the last three hold for every $i$:
\begin{equation}\label{eq:jacobian-system}
\begin{gathered}
\sum_i\lambda_i\mathrm D H(v_i)=0,
\qquad
\sum_i\lambda_iG_i\mathrm D T_i=G,\\
G_i\mathrm D T_i\in\Sym_n^+,
\qquad
\det\mathrm D T_i>0,\\
\bar\rho(x)=\rho_i(T_i(x))\det \mathrm D T_i(x).
\end{gathered}
\end{equation}
\end{lemma}

\begin{proof}
Subtracting $H(0)$ changes neither the barycenters nor the optimal
couplings, so we may assume $H\ge H(0)=0$.
The function $H$ is $C^2$, coercive, and has a positive-definite Hessian.
By \cite[Proposition~2.5]{BrizziFrieseckeRied}, the chosen barycenter is
the image of an optimal multi-marginal plan under $b_H$.
Since $b_H$ is continuous, $\bar\mu$ is compactly supported.
Since the marginals are compactly supported and absolutely continuous,
\cite[Theorem~5.7]{BrizziFrieseckeRied} gives $\bar\mu\ll\cL^n$.
The uniform Hessian bounds also imply conditions (H1)--(H3) in
\cite{GangboMcCann}.  By \cite[Theorem~4.5]{GangboMcCann}, there is a unique optimal map $T_i$
from $\bar\mu$ to each $\mu_i$.  Applying the same theorem in the reverse
direction gives an inverse defined $\mu_i$-almost everywhere, so $T_i$ is injective
$\bar\mu$-almost everywhere.

Choose a $c$-concave Kantorovich potential $\phi_i$ for $T_i$ by taking
the $c$-transform over the compact set $\supp\mu_i$
\cite[Remark~4.6]{GangboMcCann}.
The $c$-transform expresses $\phi_i$ as an infimum of functions whose
Hessians are bounded above by $bI$, so $\phi_i$ is locally
$b$-semiconcave; see also \cite[Proposition~C.2]{GangboMcCann}.
Alexandrov's theorem
\cite[Theorem~14.25]{Villani2009} gives a second-order expansion and
approximate differentiability of $\mathrm D\phi_i$ at $\cL^n$-almost every
point.  Choose the representative
$T_i(x)=x-(\mathrm D H)^{-1}(\mathrm D\phi_i(x))$ where
$\mathrm D\phi_i(x)$ exists, and set $T_i(x)=0$ on the remaining
$\cL^n$-null set.  Since $(\mathrm D H)^{-1}$ is $C^1$, this representative
is approximately differentiable $\cL^n$-almost everywhere.
It is therefore approximately differentiable $\bar\mu$-almost everywhere,
because $\bar\mu\ll\cL^n$.  The chain rule gives
\begin{equation}\label{eq:potential-map-hessian}
 \mathrm D\phi_i=\mathrm D H(v_i),\qquad
 \mathrm D^2\phi_i=G_i(I-\mathrm D T_i)
 \quad\bar\mu\text{-a.e.}
\end{equation}
The defining inequality for $\phi_i$ implies that
$z\mapsto H(z-T_i(x))-\phi_i(z)$ has a minimum at $x$ for
$\bar\mu$-almost every $x$.  Thus $\mathrm D^2\phi_i\le G_i$, and
$G_i\mathrm D T_i=G_i-\mathrm D^2\phi_i$ is symmetric and positive
semidefinite.

The measure $(\operatorname{id},T_1,\ldots,T_k)_\#\bar\mu$ glues the
optimal couplings over their common marginal $\bar\mu$.
The barycenter and multi-marginal problems have the same optimal value
\cite[Proposition~2.5]{BrizziFrieseckeRied}.  Hence the nonnegative
difference between $\sum_i\lambda_iH(x-T_i(x))$ and
$\min_z\sum_i\lambda_iH(z-T_i(x))$ has integral zero.
Thus $x=b_H(T_1(x),\ldots,T_k(x))$ for $\bar\mu$-almost every $x$.
The first-order optimality condition gives
$\sum_i\lambda_i\mathrm D H(v_i)=0$.
Hence $\sum_i\lambda_i\mathrm D\phi_i=0$ $\cL^n$-almost everywhere
on $E=\{\bar\rho>0\}$.
At $\cL^n$-almost every density point of $E$, all the gradients
$\mathrm D\phi_i$ are approximately differentiable.  Their weighted sum
vanishes on a set of density one, so its approximate differential is zero.
Thus $\sum_i\lambda_i\mathrm D^2\phi_i=0$ $\cL^n$-almost everywhere on $E$, as in
the proof of \cite[Theorem~3.13]{MaWassersteinBarycenter}.
Equation~\eqref{eq:potential-map-hessian} now gives
$\sum_i\lambda_iG_i\mathrm D T_i=G$.

For each $i$, the map $T_i$ is injective on a Borel set of full
$\bar\mu$-measure and approximately differentiable $\cL^n$-almost everywhere.
Both $\bar\mu$ and
$\mu_i=(T_i)_\#\bar\mu$ are absolutely continuous.  The Jacobian equation in
\cite[Theorem~11.1]{Villani2009} (see also
\cite[3.1.8 and 3.2.3]{Federer}) therefore gives
$\bar\rho=(\rho_i\circ T_i)|\det\mathrm D T_i|$ $\bar\mu$-almost everywhere.
The determinant is nonnegative
because $G_i\mathrm D T_i$ is positive semidefinite and $G_i>0$.
Since $\bar\rho>0$ on $E$ and $\rho_i\circ T_i$ is finite
$\bar\mu$-almost everywhere, the identity gives $\det\mathrm D T_i>0$.
Thus $G_i\mathrm D T_i$ is positive definite.
\end{proof}

\Needspace{5\baselineskip}
\begin{proposition}[Entropy inequality for a smooth cost]\label{prop:smooth-entropy}
Under the hypotheses of Lemma~\ref{lem:jacobian-identities}, suppose in
addition that
\begin{equation}\label{eq:H-defect-bound}
\log\det\left(\sum_i\lambda_i\mathrm D^2H(z_i)\right)
-\sum_i\lambda_i\log\det \mathrm D^2H(z_i)
\le c,
\end{equation}
for every $(z_1,\ldots,z_k)\in(\R^n)^k$, with the weights fixed in
Lemma~\ref{lem:jacobian-identities}.  Then
\begin{equation}\label{eq:smooth-entropy-bound}
\Ent_{\cL^n}(\bar\mu)
\le\sum_i\lambda_i\Ent_{\cL^n}(\mu_i)+c.
\end{equation}
\end{proposition}

\begin{proof}
We use the determinant argument for barycenter entropy inequalities
\cite[Section~7]{AguehCarlier}; see also \cite[Theorem~4.6]{KimPass}.
At a point where the conclusions of
Lemma~\ref{lem:jacobian-identities} hold, define
\(A_i=G^{-1/2}(G_i\mathrm D T_i)G^{-1/2}\).
Then \(A_i>0\) and \(\sum_i\lambda_iA_i=I\).
Concavity of \(A\mapsto(\det A)^{1/n}\)
\cite[Section~4.2, (4.21)]{BhatiaPositive} gives
\[
1\ge\sum_i\lambda_i
\left(\frac{\det G_i}{\det G}\det \mathrm D T_i\right)^{1/n}.
\]
Apply the last equality in \eqref{eq:jacobian-system} and the weighted
arithmetic--geometric mean inequality.  Taking logarithms gives
\[
\log\bar\rho(x)
\le\sum_i\lambda_i\log\rho_i(T_i(x))
+\log\det G-\sum_i\lambda_i\log\det G_i.
\]
The difference
$\log\det G-\sum_i\lambda_i\log\det G_i$ is at most \(c\) by
\eqref{eq:H-defect-bound}.
The measure $\bar\mu$ is compactly supported by
Lemma~\ref{lem:jacobian-identities}.
Since $(T_i)_\#\bar\mu=\mu_i$ and the marginals have finite entropy,
the right-hand side is integrable with respect to $\bar\mu$.
Integrating the pointwise inequality therefore gives 
\eqref{eq:smooth-entropy-bound}.
\end{proof}

\subsubsection{Passage to the Minkowski cost}

Proposition~\ref{prop:smooth-entropy}, applied to
$H=h_\varepsilon$, together with Lemma~\ref{lem:defect-smoothing},
gives, for compactly supported marginals with finite entropy,
\begin{equation}\label{eq:eps-compact-entropy}
\Ent_{\mathcal L^n}(\bar\mu_\varepsilon)
\le
\sum_i\lambda_i\Ent_{\mathcal L^n}(\mu_i)+c_F(F).
\end{equation}
For arbitrary $\mu_i\in\mathrm D(\Ent_{\mathcal L^n})$, we obtain a
barycenter satisfying the same inequality by truncation.
For all sufficiently large $R$, so that $\mu_i(B_R(0))>0$ for every $i$, set
$\mu_i^R:=\mu_i|_{B_R(0)}/\mu_i(B_R(0))$.
Then
\[
W_2(\mu_i^R,\mu_i)\to0,
\qquad
\Ent_{\mathcal L^n}(\mu_i^R)
\to
\Ent_{\mathcal L^n}(\mu_i).
\]
Take a sequence of optimal multi-marginal plans with marginals $(\mu_i^R)$.
These plans are relatively compact in $W_2$ on the product space because
their marginals converge in $W_2$.  To verify optimality of a limit, glue
any competing plan with marginals $(\mu_i)$ to optimal couplings from
each $\mu_i$ to $\mu_i^R$.  The resulting plans converge in $W_2$, and all
the costs converge by continuity and quadratic growth.  Passing to the
limit in the optimality inequality shows that the limiting plan is optimal;
compare the
two-marginal stability theorem \cite[Theorem~5.20]{Villani2009}.
The barycenters are the images of these plans under $b_{h_\varepsilon}$.
This map is Lipschitz by Lemma~\ref{lem:pointbar-stability}, so the
barycenters also converge in $W_2$.  Continuity of $\mathcal T_{h_\varepsilon}$
under $W_2$-convergence preserves their minimizing property
\cite[Theorems~5.20 and~6.9, Definition~6.8(iv)]{Villani2009}.
Lower semicontinuity of entropy therefore gives an
$h_\varepsilon$-Wasserstein barycenter satisfying
\eqref{eq:eps-compact-entropy} for arbitrary marginals with finite entropy.

\begin{theorem}[Log-determinant upper bound]\label{thm:minkowski-upper}
If \(F\) satisfies \eqref{eq:ellipticity}, then
\[
(\R^n,F,\cL^n)\text{ satisfies }\BCD_{c_F(F)}(0,\infty).
\]
\end{theorem}

\begin{proof}
Fix marginals $\mu_i\in\mathrm D(\Ent_{\cL^n})$.
For every \(\varepsilon>0\), choose an \(h_\varepsilon\)-Wasserstein barycenter
\(\bar\mu_\varepsilon\) satisfying
\begin{equation}\label{eq:eps-ent-final}
\Ent_{\cL^n}(\bar\mu_\varepsilon)
\le\sum_i\lambda_i\Ent_{\cL^n}(\mu_i)+c_F(F).
\end{equation}
Let \(\gamma_\varepsilon\) be the associated optimal multi-marginal plan in the
construction above, so that
\(\bar\mu_\varepsilon=(b_{h_\varepsilon})_\#\gamma_\varepsilon\).
The marginals are fixed, so the plans are tight and their weighted second
moments equal $\sum_i\lambda_i\int|x|^2\d\mu_i$, independently of
$\varepsilon$.  Every weak limit has the same marginals and hence the same
weighted second moment.  By \cite[Theorem~6.9]{Villani2009}, they are therefore
relatively compact in the $W_2$ topology on the product space.  Along a
subsequence, \(\gamma_\varepsilon\to\gamma\) in this topology.
Lemma~\ref{lem:pointbar-stability} gives
\(\bar\mu_\varepsilon\to\bar\mu:=(b_h)_\#\gamma\) in \(W_2\).

Let
\[
J_\varepsilon(\nu)=\sum_i\lambda_i\mathcal T_{h_\varepsilon}(\nu,\mu_i),
\qquad
J(\nu)=\sum_i\lambda_i\mathcal T_h(\nu,\mu_i).
\]
By Lemma~\ref{lem:cost-approx},
\(\sup_{\nu\in\Pp(\R^n)}|J_\varepsilon(\nu)-J(\nu)|
\le C_0\varepsilon^2\).
Since \(\bar\mu_\varepsilon\) minimizes \(J_\varepsilon\), it is a
\(2C_0\varepsilon^2\)-minimizer of \(J\).  Since $h$ is continuous and has
quadratic growth, $J$ is continuous under \(W_2\)-convergence
\cite[Theorems~5.20 and~6.9, Definition~6.8(iv)]{Villani2009}.  Therefore
\(\bar\mu\) minimizes \(J\).  

Since \(h=F^2/2\), the measure $\bar\mu$ is a
Wasserstein barycenter for the metric induced by $F$.  Lower semicontinuity of entropy
and \eqref{eq:eps-ent-final} give
\[
\Ent_{\cL^n}(\bar\mu)
\le\sum_i\lambda_i\Ent_{\cL^n}(\mu_i)+c_F(F).
\]
\end{proof}

\begin{proposition}[Rigidity of the log-determinant defect]\label{prop:cf-rigidity}
For a smooth strongly convex Minkowski norm $F$,
\[
c_F(F)=0
\quad\Longleftrightarrow\quad
F\text{ is induced by an inner product}.
\]
\end{proposition}

\begin{proof}
For a norm induced by an inner product, \(g_v\) is independent of \(v\), so
\(c_F(F)=0\).  Conversely, suppose that \(c_F(F)=0\).  Strict concavity of
$\log\det$, applied to \((g_v+g_w)/2\), gives \(g_v=g_w\) for all nonzero
\(v,w\).  

Thus \(\mathrm D^2h\) is constant off the origin.  By
homogeneity of degree two, \(h\) is a quadratic form.
\end{proof}

\subsection{Lower bounds and perturbations of the Euclidean norm}

For a normed space, the cotangent norm is the dual norm at every point.
Theorem~\ref{thm:quantitative-hilbertianity} therefore gives the following
estimates.

\begin{corollary}[Quantitative estimates for normed spaces]\label{thm:reverse}
If $(\R^n,F,\cL^n)$ satisfies \(\BCD_\delta(0,\infty)\), then
\begin{align}
\mathfrak p_{\mathrm{par}}(F^*)&\le8\sqrt\delta,
\label{eq:reverse-pstar}\\
\Delta_{\mathrm{BCD}}(F)&\ge\frac1{64}\mathfrak p_{\mathrm{par}}(F^*)^2.
\label{eq:Delta-lower}
\end{align}
For every $p,q\in(\R^n)^*$, not both zero,
\[
\left|
\frac{F^*(p+q)^2+F^*(p-q)^2}
{2F^*(p)^2+2F^*(q)^2}-1
\right|
\le8\sqrt\delta.
\]
In particular, the von Neumann--Jordan constant \cite{ClarksonNJ}
of the dual norm satisfies
\[
 C_{\mathrm{NJ}}(F^*)
 :=\sup_{\substack{p,q\in(\R^n)^*\\(p,q)\ne(0,0)}}
 \frac{F^*(p+q)^2+F^*(p-q)^2}{2F^*(p)^2+2F^*(q)^2}
 \le1+8\sqrt\delta.
\]
\end{corollary}

\begin{proof}
Theorem~\ref{thm:quantitative-hilbertianity}, applied to the constant
cotangent norm, gives \eqref{eq:reverse-pstar}.  Taking the infimum over all
admissible $\delta$ gives \eqref{eq:Delta-lower}.  The bound for each pair
$p,q$ follows by dividing the parallelogram expression by
$2F^*(p)^2+2F^*(q)^2$.  This denominator is at least the maximum appearing
in the definition of $\mathfrak p_{\mathrm{par}}(F^*)$.
Taking the supremum gives the bound for $C_{\mathrm{NJ}}(F^*)$.
\end{proof}

The quadratic lower bound in \eqref{eq:Delta-lower} comes from optimizing
the parameter $\varepsilon$ in the proof of
Theorem~\ref{thm:quantitative-hilbertianity}.  The perturbations of the
Euclidean norm below attain the same order.

Let $\psi\in C^\infty(S^{n-1})$ be even and define
$H(v)=|v|^2\psi(v/|v|)$ for $v\ne0$, with $H(0)=0$.
For sufficiently small $|\tau|$, set
\[
F_\tau(v)^2=|v|^2+\tau H(v).
\]
Then $F_\tau$ is a smooth strongly convex Minkowski norm.  For $\tau\ne0$,
it is induced by an inner product if and only if $H$ is a quadratic form.
From now on, assume that $\psi$ is not the restriction of a quadratic form.

Set
\(k(v)=H(v)/2\) and \(A_v=\mathrm D^2k(v)\) for \(|v|=1\).  Then the
fundamental tensor is \(g_v^\tau=I+\tau A_v\).
Define the maximal variance of $A_v$ in the Hilbert--Schmidt norm by
\[
\mathcal V(A):=
\sup_{\sigma\in\cP(S^{n-1})}
\left\{
\int\tr(A_v^2)\d\sigma(v)
-\tr\!\left[\left(\int A_v\d\sigma(v)\right)^2\right]
\right\}.
\]

\begin{proposition}[Second-order expansion of the log-determinant defect]\label{prop:cf-expansion}
As \(\tau\to0\),
\begin{equation}\label{eq:cf-expansion}
c_F(F_\tau)
=\frac{\tau^2}{2}\mathcal V(A)+O(|\tau|^3).
\end{equation}
The remainder is uniform in the number of directions and in the weights.
Moreover,
\[
\mathcal V(A)=0
\quad\Longleftrightarrow\quad
H\text{ is a quadratic form}.
\]
\end{proposition}

\begin{proof}
	Since \(k\) is homogeneous of degree two,
	$\mathrm D^2k(rv)=\mathrm D^2k(v)$ for $r>0$.
	Thus we may normalize all directions in the definition of
	\(c_F(F_\tau)\) to lie in \(S^{n-1}\).
	The set $\mathcal A:=\{A_v:v\in S^{n-1}\}$ and its convex hull
	are compact.  Hence, uniformly
	for \(B\in\operatorname{co}(\mathcal A)\),
	\begin{equation}\label{eq:logdet-expansion}
		\log\det(I+\tau B)
		=
		\tau\tr B-\frac{\tau^2}{2}\tr(B^2)+O(|\tau|^3).
	\end{equation}
	
	Fix unit vectors \(v_1,\ldots,v_m\) and positive weights
	\(\lambda_i\) with \(\sum_i\lambda_i=1\).  Write
	$A_i:=A_{v_i}$ and $\bar A:=\sum_i\lambda_iA_i$.
	Since \(g_{v_i}^\tau=I+\tau A_i\), applying
	\eqref{eq:logdet-expansion} to \(\bar A\) and to each \(A_i\) gives
	\begin{align*}
		&\log\det\left(\sum_i\lambda_i g_{v_i}^\tau\right)
		-\sum_i\lambda_i\log\det g_{v_i}^\tau
		\\
		&\qquad=
		\frac{\tau^2}{2}
		\left[
		\sum_i\lambda_i\tr(A_i^2)
		-\tr(\bar A^2)
		\right]
		+O(|\tau|^3).
	\end{align*}
	The remainder is uniform in \(m\), the directions, and the weights,
	because \(\bar A,A_1,\ldots,A_m\) all belong to the fixed compact set
	\(\operatorname{co}(\mathcal A)\) and \(\sum_i\lambda_i=1\).
	
	For $\sigma=\sum_i\lambda_i\delta_{v_i}$, the expression in square
	brackets is
	\[
	\int\tr(A_v^2)\d\sigma(v)
	-\tr\!\left[
	\left(\int A_v\d\sigma(v)\right)^2
	\right].
	\]
	Taking the supremum over all finite families therefore yields
	\[
	c_F(F_\tau)
	=
	\frac{\tau^2}{2}
	\sup_{\substack{\sigma\in\cP(S^{n-1})\\
			\sigma\ \mathrm{finitely\ supported}}}
	\left\{
	\int\tr(A_v^2)\d\sigma(v)
	-\tr\!\left[
	\left(\int A_v\d\sigma(v)\right)^2
	\right]
	\right\}
	+O(|\tau|^3).
	\]
	The functional inside the supremum is continuous under weak convergence
	of probability measures.  Since finitely supported probability measures
	are weakly dense in \(\cP(S^{n-1})\), the supremum is precisely
	\(\mathcal V(A)\).  This proves \eqref{eq:cf-expansion}.
	
	Finally, for \(v,w\in S^{n-1}\), take
	$\sigma=(\delta_v+\delta_w)/2$.  Then
	\[
	\int\tr(A_u^2)\d\sigma(u)
	-\tr\!\left[
	\left(\int A_u\d\sigma(u)\right)^2
	\right]
	=
	\frac14\tr\bigl((A_v-A_w)^2\bigr).
	\]
	Thus \(\mathcal V(A)=0\) if and only if \(A_v\) is independent of \(v\).
	
	If \(A_v\equiv A\), homogeneity of degree two implies
	$\mathrm D^2k(x)=A$ for $x\ne0$.
	Therefore \(H=2k\) is a quadratic form.  The converse is immediate.
\end{proof}

To obtain a lower bound of the same order, define
\begin{equation}\label{eq:QH}
\mathcal Q(H):=
\sup_{(p,q)\ne(0,0)}
\frac{|H(p+q)+H(p-q)-2H(p)-2H(q)|}{M_0(p,q)},
\end{equation}
where the supremum is over $p,q\in\R^n$ and
$M_0(p,q)=\max\{|p+q|^2,|p-q|^2,2|p|^2,2|q|^2\}$.
If $\mathcal Q(H)=0$, then $H$ satisfies the parallelogram identity.
Since $H$ is continuous, even, and homogeneous of degree two, the polarization
argument of Jordan--von Neumann \cite[pp.~721--722]{JordanvonNeumann}
shows that $H$ is a quadratic form.  Thus $\mathcal Q(H)>0$ under our
assumption on $\psi$.

We identify $(\R^n)^*$ with $\R^n$ using the Euclidean inner product.

\begin{lemma}[First-order expansion of the squared dual norm]
As $\tau\to0$, the following expansion holds uniformly for $p\in(\R^n)^*$:
\begin{equation}\label{eq:dual-expansion}
F_\tau^*(p)^2
=|p|^2-\tau H(p)+O(\tau^2|p|^2).
\end{equation}
If \(H\) is not quadratic, then for sufficiently small \(|\tau|\),
\begin{equation}\label{eq:pstar-lower-tau}
\mathfrak p_{\mathrm{par}}(F_\tau^*)
\ge\frac12\mathcal Q(H)|\tau|.
\end{equation}
\end{lemma}

\begin{proof}
Let \(h_\tau(v)=\frac12|v|^2+\tau k(v)\), where \(k=H/2\).  Uniform strong
convexity, homogeneity of degree two, and the implicit equation
$p=\mathrm D h_\tau(v)=v+\tau \mathrm D k(v)$
give \(v=p-\tau \mathrm D k(p)+O(\tau^2|p|)\).  The first two derivatives of $k$
are uniformly bounded on the Euclidean unit sphere, and homogeneity of degree two
extends the estimate uniformly to all $p$.  Substitution into the Legendre
transform yields
\(h_\tau^*(p)=|p|^2/2-\tau k(p)+O(\tau^2|p|^2)\).
Since \(F_\tau^*(p)^2=2h_\tau^*(p)\), this is
\eqref{eq:dual-expansion}.  The Euclidean terms cancel in the parallelogram
expression, so
\[
\begin{aligned}
&F_\tau^*(p+q)^2+F_\tau^*(p-q)^2
-2F_\tau^*(p)^2-2F_\tau^*(q)^2\\
&\qquad
=-\tau\bigl[H(p+q)+H(p-q)-2H(p)-2H(q)\bigr]
+O(\tau^2M_0(p,q)).
\end{aligned}
\]
The denominator in $\mathfrak p_{\mathrm{par}}(F_\tau^*)$ equals
$M_0(p,q)(1+O(|\tau|))$.  Choose $p,q$ for which the
quotient in \eqref{eq:QH} is greater than $3\mathcal Q(H)/4$.  For small
\(|\tau|\), division gives \eqref{eq:pstar-lower-tau}.
\end{proof}

\begin{theorem}[Quadratic order near Euclidean norms]
\label{thm:perturbative-equivalence}
If \(H\) is not quadratic, then for sufficiently small \(|\tau|\),
\[
\frac{\mathcal Q(H)^2}{256}\tau^2
\le\Delta_{\mathrm{BCD}}(F_\tau)
\le c_F(F_\tau)
=\frac{\tau^2}{2}\mathcal V(A)+O(|\tau|^3).
\]
In particular,
$\Delta_{\mathrm{BCD}}(F_\tau)\asymp_\psi\tau^2$.
\end{theorem}

\begin{proof}
Corollary~\ref{thm:reverse} and \eqref{eq:pstar-lower-tau} give the lower
bound.  Theorem~\ref{thm:minkowski-upper} gives the middle inequality, and
Proposition~\ref{prop:cf-expansion} gives the expansion.
\end{proof}

\begin{proof}[Proof of Theorem~\ref{thm:intro-finsler}]
The two-sided estimate follows from Theorem~\ref{thm:minkowski-upper} and
Corollary~\ref{thm:reverse}.  Proposition~\ref{prop:cf-rigidity} and the
Jordan--von Neumann characterization \cite{JordanvonNeumann} show that
$c_F(F)$ and $\mathfrak p_{\mathrm{par}}(F^*)$ vanish exactly when the norm
is induced by an inner product.  The two-sided estimate then gives the
equivalence with $\Delta_{\mathrm{BCD}}(F)=0$.
The final assertion follows from
Theorem~\ref{thm:perturbative-equivalence}.
\end{proof}

\begin{example}[An explicit non-Riemannian almost BCD space]

Take $\psi(\theta)=\theta_1^4+\theta_2^4$ and $\tau=1/25$ in the
preceding perturbation family.  The coefficient $1/25$ ensures that the
log-determinant bound satisfies the sufficient condition in
Theorem~\ref{thm:enb}.  The resulting norm is
\[
F_{\mathrm{ex}}(x,y)^2
=x^2+y^2+\frac1{25}\frac{x^4+y^4}{x^2+y^2},
\qquad (x,y)\ne(0,0),
\]
with \(F_{\mathrm{ex}}(0)=0\).  Away from the origin, direct calculation gives
\(24I/25\le \mathrm D^2(F^2_{\mathrm{ex}}/2)\le11I/10\).  Thus
\(c_F(F_{\mathrm{ex}})\le2\log(55/48)<\frac12\log2\).
Theorem~\ref{thm:minkowski-upper} gives
$\BCD_{c_F(F_{\mathrm{ex}})}(0,\infty)$.
The norm is not induced by an inner product: for the coordinate unit
vectors $e_1,e_2$, the two sides of the parallelogram identity are
$F_{\mathrm{ex}}(e_1+e_2)^2+F_{\mathrm{ex}}(e_1-e_2)^2=102/25$ and
$2F_{\mathrm{ex}}(e_1)^2+2F_{\mathrm{ex}}(e_2)^2=104/25$.
\end{example}
\paragraph{Acknowledgements.}
We thank Emanuel Milman for his interest in our work on Wasserstein
barycenters and for suggesting that we investigate the equivalence between
BCD and RCD.
We also thank Luigi Ambrosio for valuable discussions and suggestions
concerning essential non-branching.

\paragraph{Statement on AI-assisted research.}
ChatGPT and Codex were used during the research and preparation of this
manuscript.  Their use included mathematical exploration, testing proof
strategies, checking arguments and references, and drafting and revising
the text.  The research questions, motivation, and conceptual framework
arose from the authors' ongoing research program.  The overall proof
strategies for the main theorems were conceived by the authors.

\end{document}